\documentclass[11pt]{amsart}
\usepackage{stmaryrd}
\usepackage{color}
\usepackage{amsfonts}
\usepackage{mathrsfs}
\usepackage{amssymb,amsmath,amsthm}
\usepackage{epsfig}
\usepackage{graphicx}
\usepackage{appendix}
\usepackage{lineno}
\usepackage{enumerate}
\usepackage{comment}
\usepackage[numbers,sort&compress]{natbib}
\newcommand{\dps}{\displaystyle}
\newtheorem{proposition}{\indent Proposition}[section]
\newtheorem{theorem}{\indent Theorem}[section]
\newtheorem{lemma}{\indent Lemma}[section]

\newtheorem{definition}{\indent Definition}[section]
\newtheorem{remark}{\indent Remark}[section]

\newtheorem{example}{\indent Example}[section]
\newcommand{\ba}{\begin{array}}\newcommand{\ea}{\end{array}}
\newcommand{\be}{\begin{eqnarray}}\newcommand{\ee}{\end{eqnarray}}
\newcommand{\beq}{\begin{equation}}\newcommand{\eeq}{\end{equation}}
\newcommand{\bex}{\begin{eqnarray*}}
\newcommand{\eex}{\end{eqnarray*}}

\def\bq{\begin{equation}}
\def\eq{\end{equation}}
\def\beq{\begin{equation*}}
\def\eeq{\end{equation*}}
\def\br{\begin{eqnarray}}
\def\er{\end{eqnarray}}
\def\brr{\bq\begin{array}{rlll}}
\def\err{\end{array}\eq}
\def\barr{\bq\begin{array}{r@{}l}}
\def\earr{\end{array}\eq}
\def\bearr{\beq\begin{array}{r@{}l}}
\def\eearr{\end{array}\eeq}
\def\brr{\bq\begin{array}{r@{}l}}
\def\err{\end{array}\eq}
\def\bry{\beq\begin{array}{r@{}l}}
\def\ery{\end{array}\eeq}

\font\tenbi=cmmib10   at 11 pt
\font\sevenbi=cmmib10 at 9pt
\font\fivebi=cmmib7 at 6pt
\newfam\bifam
\textfont\bifam=\tenbi \scriptfont\bifam=\sevenbi  \scriptscriptfont\bifam=\fivebi

\font\sixtdb=msbm10 at 16 pt \font\tendb=msbm10 at 12 pt  \font\sevendb=msbm7
\newfam\dbfam
\textfont\dbfam=\sixtdb

\textfont\dbfam=\tendb \scriptfont\dbfam=\sevendb

\title[Fractional Jacobi Spectral Methods for Singular VIEs]
{A M\"untz-Collocation spectral method for weakly singular Volterra integral equations$^*$}
\author[ D.M. Hou, Y.M. Lin, M. Azaiez, \& C.J. Xu ]
{Dianming Hou$^{1,2}$
\quad
Yumin Lin$^{2}$
\quad
Mejdi Azaiez$^{2,3}$
\quad
Chuanju Xu$^{2,3,4}$}
\thanks{\hskip -12pt
${}^*$This research is partially supported by NSF of China
(Grant numbers 
11971408, 51661135011, 11421110001, and
91630204). The third author has received financial support from the
French State in the frame of the ``Investments for the future" Programme Idex Bordeaux,
reference ANR-10-IDEX-03-02 and NSFC/ANR joint program 51661135011/ANR-16-
CE40-0026-01.\\
${}^{1}$School of Mathematics and Statistics, Jiangsu Normal
University, 221116 Xuzhou, China.\\
${}^{2}$School of Mathematical Sciences and
Fujian Provincial Key Laboratory of Mathematical Modeling and High Performance
Scientific Computing, Xiamen
University, 361005 Xiamen, China.\\
${}^{3}$Bordeaux INP, Laboratoire I2M UMR 5295, 33607 Pessac, France.\\
${}^{4}$Corresponding author. Email: cjxu@xmu.edu.cn}


\begin{document}
\graphicspath{{figures/},}
\date {\today}
\maketitle
\begin{abstract}
In this paper we propose and analyze a fractional Jacobi-collocation spectral method for the second kind Volterra integral equations (VIEs)
with weakly singular kernel $(x-s)^{-\mu},0<\mu<1$. First we develop a family of fractional Jacobi polynomials, along with
basic approximation results for some weighted projection and interpolation operators defined in suitable weighted Sobolev spaces.
Then we construct an efficient fractional Jacobi-collocation spectral method for the VIEs
using the zeros of the new developed fractional Jacobi polynomial.
A detailed convergence analysis is carried out to derive error estimates of the numerical solution in both
$L^{\infty}$- and weighted $L^{2}$-norms.
The main novelty of the paper is that the proposed method is highly efficient for typical solutions that VIEs
usually possess. Precisely, it is proved that the exponential convergence rate can be achieved for solutions which are smooth
after the variable change $x\rightarrow x^{1/\lambda}$ for a suitable real number $\lambda$.
Finally a series of numerical examples are presented to demonstrate the efficiency of the method.
\end{abstract}

\section{Introduction}
Among various methods, spectral methods have proven to be one of the most efficient approaches
for solving partial differential equations.
The efficiency of spectral methods, however, depends crucially on the regularity of the solution.
The convergence of the spectral method is exponential, provided that the solution of the problem is sufficiently smooth. On the other word, the traditional spectral methods will lose high accuracy
when facing  problems with solutions of limited regularity.
Such problems include the second kind Volterra integral equations (VIEs), which we consider in this paper.
It has been well known that the solution of
the VIEs with weakly singular kernel is of low regularity at the boundaries of the domain.

Volterra integral equations model evolutionary problems with memory, which can be found in a number of  disciplines, such as electromagnetic scattering, demography, viscoelastic materials, insurance mathematics, etc. They have been subject of many theoretical and numerical investigations.
In this paper we consider the following integral equation:
\begin{equation*}
u(x)=g(x)+\int_{0}^{x}(x-s)^{-\mu}K(x,s)u(s)ds,~~~0<\mu<1,~~x\in I:=[0,1],
\end{equation*}
where the source function $g(x)\in C(I)$, and $K\in C(I\times I)$ with $K(x,x)\neq 0$ for $x\in I$.
Our aim is to design a numerical method for this equation, which will be shown
more efficient than the existing methods
in capturing the singularities of typical solutions of this kind of equations.

In fact there exists a considerable amount of numerical methods in the literature
for the equation under consideration; see, e.g., \cite{SSW15,WS16,SWG14,HM16,Bru83,Bru85,Bru04,Che13,Chen09,Che10,Tang94,Li12,Li15} and the references therein.
In the framework of high order methods for example,
a Legendre-collocation method, along with rigorous error analysis, was proposed in \cite{Tang08,Tang10}
for VIEs with smooth kernel (i.e., $\mu=0$) and regular source term; There have also been
spectral methods \cite{Ali08,Ali09}, used to approximate smooth solutions of delay differential or
integral equations with two or more vanishing delays. In \cite{Chen09,Che13},
Chen et al. proposed and analyzed a spectral Jacobi-collocation method for VIEs under assumption
that the underlying solutions are sufficiently smooth.
However, as demonstrated in \cite{Bru04}, VIEs with weakly singular kernel typically have solutions whose derivatives are unbounded at the left end point, even when the source term $g(x)$ is regular. Partially for this reason, the numerical treatment of the VIEs is not simple. In particular, numerical methods based on approximations of high order polynomials, such as spectral methods or p-version finite element methods, fail to yield high order convergence since increasing the polynomial degree does not improve the approximation accuracy
for low regular functions. Nevertheless, efforts have been made to handle this problem in some different circumstances.
Let us mention, among others,
Chen and Tang \cite{Che10} for a transformation method to transform the equation into a new
Volterra integral equation defined on the standard interval $[-1,1]$, so that the solution of the new equation possesses better regularity.
Recently, Li, Tang, and Xu \cite{Li15} have constructed a Chebyshev and Legendre pseudo-spectral Galerkin method for weakly singular
VIEs based on a new function transformation to make the solution smooth.
Allaei et al.  \cite{All16} used a transformation of the independent variable for the Volterra integral equation
with weakly singular kernel, then constructed a Jacobi collocation method for the transformed equation.
It is worthwhile to emphasize that the efficiency of the aforementioned approaches depends on specific form
of the exact solution and the assumption of sufficient smoothness of the source term.
Very recently, Hou and Xu \cite{Hou17} proposed a general framework using fractional polynomials
for some weakly singular integro-differential equations and
fractional differential equations. They showed that the convergence rate of the fractional spectral methods
is closely related to the regularity of the transformation of the exact solution by the variable change
$x\rightarrow x^{1/\lambda}$, where $\lambda$ is a suitable parameter
appearing in the fractional approximation space.

This paper aims at designing, developing, and testing a fractional Jacobi spectral method for the weakly singular VIEs,
which has the capability to capture the limited regular solution in a more efficient way. The new method will make use of
the fractional Jacobi polynomial $J^{\alpha,\beta,\lambda}_{N+1}(x)$, recently introduced
in \cite{Hou17,Hou18} to deal with some singular problems.
The advantage of the proposed approach is that the exponential convergence can be guaranteed for solutions, which are smooth after
the variable change $x\rightarrow x^{1/\lambda}$ for suitable parameter $\lambda$.
It is worth to emphasize that
the present method is different from the one proposed in \cite{Hou17}.
Firstly, the method in \cite{Hou17} was for the integral equation:
$
u_{x}=a_{1}u(x)+a_{2}~{_0}I_{x}^{\mu}u(x)+f(x),
$
where $a_{1}$ and $a_{2}$ are constants, which were assumed to satisfy some
conditions in order to guarantee the well-posedness of the discrete problem; see Theorem 4.1 in \cite{Hou17}.
Secondly, the present method is of Collocation type while the method in \cite{Hou17} was constructed under
Galerkin or Petrov-Galerkin framework. It is known that the method of Collocation type is easier to implement (usually
more difficult to analyze) than the Galerkin type.

Below are the main ingredients of the paper.

$\bullet$ First, one key point of the study
is the new $\lambda-$polynomial space constructed as the approximation space. Some new derivatives are defined
such that the fractional Jacobi polynomials can inherit some desirable properties from classical Jacobi polynomials.
We also set up the relationship between the new defined derivatives and the classical ones.
Then we derive the approximating results for some projection and interpolation operators in weighted Sobolev spaces,
This is the main content of Section 2.

$\bullet$  In Section 3, the fractional Jacobi spectral-collocation method is constructed for weakly singular VIEs.
A rigorous convergence analysis is carried out for the proposed method and
some optimal error estimates in $L^{\infty}-$ and $L^{2}_{\omega^{\alpha,\beta,\lambda}}-$norms are derived.

$\bullet$  Some numerical examples are presented in Section 4 to validate the theoretical prediction.
Finally we give some concluding remarks in the last section.

 \section{Preliminaries}
 \setcounter{equation}{0}
 In this section, we will define the fractional Jacobi polynomials and study their fundamental properties.
 Especially we will introduce some weighted projection and interpolation operators and derive optimal error estimates for these operators
 in different weighted Sobolev spaces.
 These results play a key role in the convergence analysis for the numerical method that we design for the VIEs later on.

 \subsection{Some basic properties of fractional Jacobi polynomials}

We begin by defining the $\lambda-$polynomial space as follows:
\beq
P^{\lambda}_{n}(\mathbb{R}^{+}):=span\big\{1,x^{\lambda},x^{2\lambda},\cdots,x^{n\lambda}\big\},
\eeq
where $\mathbb{R}^{+}=[0,+\infty), 0<\lambda\leq 1.$

A $\lambda-$polynomial of degree $n$ is denoted by
\beq
p^{\lambda}_{n}(x):=k_{n}x^{n\lambda}+k_{n-1}x^{(n-1)\lambda}+\cdots+k_{1}x^{\lambda}+k_{0},\ \ \ k_{n}\neq0, x\in \mathbb{R}^{+},
\eeq
where $\big\{k_{i}\big\}_{i=0}^{n}$ are real constants and $k_{n}$ is called the leading coefficient of $p^{\lambda}_{n}.$
{\color{black} Hereafter, we simply denote the degree of  $\lambda-$polynomial
$p^{\lambda}_{n}(x)$ by $\deg(p^{\lambda}_{n})$.}

Let $I:=[0,1]$ and $\omega(x)\in L^{1}(I)$ be a positive weight function. A sequence of $\lambda-$polynomials $\big\{p^{\lambda}_{n}\big\}_{n=0}^{\infty}$ with $\deg(p^{\lambda}_{n})=n$ is said to be orthogonal in $L^{2}_{\omega}(I)$ if
\beq
(p^{\lambda}_{n},p^{\lambda}_{m})_{\omega}=\int^{1}_{0}p^{\lambda}_{n}(x)p^{\lambda}_{m}(x)\omega(x)dx=\gamma_{n}\delta_{m,n},
\eeq
where $\gamma_{n}=\|p_{n}^{\lambda}\|^{2}_{{\color{black}0},\omega}{\color{black}:=(p^{\lambda}_{n},p^{\lambda}_{n})_{\omega}}$, and $\delta_{m,n}$ is the Kronecker delta.

We define the space
\beq
P^{\lambda}_{n}(I):=span\big\{p^{\lambda}_{0},p^{\lambda}_{1},\cdots,p^{\lambda}_{n}\big\}.
\eeq

The following two lemmas can be easily proved by following the standard way to prove the existence of
the classical orthogonal polynomials;
see, e.g., \cite[p48-50]{STW10}.

\begin{lemma}\label{plem3}
$p^{\lambda}_{n+1}(x)$ is $\omega(x)$-weighted orthogonal to
any $\lambda-$polynomial $q(x)\in P^{\lambda}_{n}(I)$.
\end{lemma}
\begin{lemma}\label{plem1}
For any given positive weight function $\omega\in L^{1}(I)$, there exists a unique sequence of monic orthogonal $\lambda-$polynomials $\big\{\bar{p}^{\lambda}_{n}\big\}_{n=0}^{\infty}$ with $\deg(\bar{p}^{\lambda}_{n})=n$.
This unique sequence can be obtained through the following recurrence relation
\beq
\bar{p}^{\lambda}_{0}=1,\ \ \ \bar{p}^{\lambda}_{1}=x^{\lambda}-\alpha_{0},\ \ \ \bar{p}^{\lambda}_{n+1}=(x^{\lambda}-\alpha_{n})\bar{p}^{\lambda}_{n}-\beta_{n}\bar{p}^{\lambda}_{n-1},\ \ n\geq1,
\eeq
where
\beq
\alpha_{n}=\dps\frac{(x^{\lambda}\bar{p}^{\lambda}_{n},\bar{p}^{\lambda}_{n})_{\omega}}{\|\bar{p}^{\lambda}_{n}\|^{2}_{{\color{black}0,}\omega}},\ \ n\geq0;\ \ \ \ \beta_{n}=\dps\frac{\|\bar{p}_{n}\|^{2}_{{\color{black}0,}\omega}}{\|\bar{p}^{\lambda}_{n-1}\|^{2}_{{\color{black}0,}\omega}},\ \ n\geq1.
\eeq
\end{lemma}

Now we turn to define the fractional Jacobi polynomials.
\begin{definition}\label{Jabl}
The fractional Jacobi polynomials of degree $n$ are defined by
\begin{equation}\label{eq2}
J^{\alpha,\beta,\lambda}_{n}(x)=J^{\alpha,\beta}_{n}(2x^{\lambda}-1), \ \forall x \in I,
\end{equation}
where $J^{\alpha,\beta}_{n}(x)$ denotes the
Jacobi polynomial of degree n, and $\alpha, \beta>-1, 0<\lambda\leq1$.
\end{definition}
When $\lambda=1,$ the polynomials $\{J^{\alpha,\beta,1}_{n}(x)\}_{n=0}^{\infty}$ are called shifted Jacobi polynomials up to a constant, which are orthogonal polynomials with the weight $(1-x)^{\alpha}x^{\beta}.$

It has been well known that the classical Jacobi polynomial $J^{\alpha,\beta}_{n}(x)$ has the following representation
\bex
J^{\alpha,\beta}_{n}(x)=\dps\frac{\Gamma(n+\alpha+1)}{n!\Gamma(n+\alpha+\beta+1)}\sum_{k=0}^{n}\binom n k
\frac{\Gamma(n+k+\alpha+\beta+1)}{\Gamma(k+\alpha+1)}\big(\frac{x-1}{2}\big)^{k}.
\eex
Consequently, we have
\bq\label{eqx5}
J^{\alpha,\beta,\lambda}_{n}(x)=\dps\frac{\Gamma(n+\alpha+1)}{n!\Gamma(n+\alpha+\beta+1)}\sum_{k=0}^{n}\binom n k\frac{\Gamma(n+k+\alpha+\beta+1)}{\Gamma(k+\alpha+1)}(x^{\lambda}-1)^{k}.
\eq

In order to inherit important properties from the classical Jacobi polynomials, we modify the definition of
the derivatives as follows.
\begin{definition}\label{def1}
The first order new derivative is defined by
\bex
D^{1}_{\lambda}v(x)=\dps\frac{d}{dx^{\lambda}}v(x)=\frac{x^{1-\lambda}}{\lambda}v'(x),
\eex
and the new defined derivative of order $k, k\ge 1$ is denoted by
\bex
D^{k}_{\lambda}v(x)=\overbrace{D^{1}_{\lambda}\cdot D^{1}_{\lambda}\cdots D^{1}_{\lambda}}^{k}v(x),
\eex
where $0<\lambda\leq1$ and $v(x)$ is defined on $\mathbb{R}^{+}$.
\end{definition}
Furthermore, the left-side (right-side) limit definitions of the new derivatives are denoted by
\bearr
^{+}D^{1}_{\lambda}v(x)&:=\dps\lim_{\Delta x\rightarrow 0^{+}}\frac{v(x+\Delta x)-v(x)}{(x+\Delta x)^{\lambda}-x^{\lambda}},\\[9pt]
^{-}D^{1}_{\lambda}v(x)&:=\dps\lim_{\Delta x\rightarrow 0^{-}}\frac{v(x+\Delta x)-v(x)}{(x+\Delta x)^{\lambda}-x^{\lambda}}.\\
\eearr
$D^{1}_{\lambda}v(x)$ exists if and only if $^{+}D^{1}_{\lambda}v(x)={}^{-}D^{1}_{\lambda}v(x)$, and
$D^{1}_{\lambda}v(x)={}^{+}D^{1}_{\lambda}v(x)={}^{-}D^{1}_{\lambda}v(x).$
Obviously when $\lambda=1$, the new defined derivatives become the classical ones.

We set two weight functions as follows:
\be\label{oabl}
\omega^{\alpha,\beta,\lambda}(x):=\lambda(1-x^{\lambda})^{\alpha}x^{(\beta+1)\lambda-1}.
\ee
\be\label{oab2}
\widehat{\omega}^{\alpha,\beta,\lambda}(x):=(1-x^{\lambda})^{\alpha}x^{\beta\lambda}=\lambda^{-1} x^{1-\lambda}\omega^{\alpha,\beta,\lambda}(x).
\ee
\begin{lemma}[see \cite{Hou17,Hou18}]\label{lem8}
The fractional Jacobi polynomials $J^{\alpha,\beta,\lambda}_{n}(x)$ are orthogonal with respect to the weight function
$\omega^{\alpha,\beta,\lambda}(x)$,
$\alpha, \beta>-1,0<\lambda\leq1$, i.e.,
\begin{equation}\label{eq12}
\int_{0}^{1}\omega^{\alpha,\beta,\lambda}(x)J^{\alpha,\beta,\lambda}_{n}(x)J^{\alpha,\beta,\lambda}_{m}(x)dx=\widehat{\gamma}_{n}^{\alpha,\beta}\delta_{m,n},
\end{equation}
where \begin{equation*}
\widehat{\gamma}_{n}^{\alpha,\beta}=\frac{\Gamma(n+\alpha+1)\Gamma(n+\beta+1)}{(2n+\alpha+\beta+1)n!\Gamma(n+\alpha+\beta+1)}.
\end{equation*}
\end{lemma}
The special case $\alpha=0,\beta=\frac{1}{\lambda}-1$ yields the  M\"{u}ntz Legendre polynomials, which have been investigated in a different way in \cite{MSS93,BEZ94}.

Next, we shall show that the fractional Jacobi polynomials are the eigenfunctions of a singular Sturm-Liouville operator $\mathcal{L}^{\alpha,\beta}_{\lambda}$ defined by
\bearr
\mathcal{L}^{\alpha,\beta}_{\lambda}v(x)&= -(\widehat{\omega}^{\alpha,\beta,\lambda}(x))^{-1}D^{1}_{\lambda}\big\{(1-x^{\lambda})^{\alpha+1}x^{(\beta+1)\lambda}D^{1}_{\lambda}v(x)\big\}.
\eearr
\begin{lemma}[see \cite{Hou17,Hou18}]\label{lemma4}
The fractional Jacobi polynomials $\big\{J^{\alpha,\beta,\lambda}_{n}\big\}_{n=0}^{\infty}$
satisfy the following singular Sturm-Liouville problem:
\begin{equation}\label{Eq5}
 \mathcal{L}^{\alpha,\beta}_{\lambda}J^{\alpha,\beta,\lambda}_{n}(x)=\sigma^{\alpha,\beta}_{n}J^{\alpha,\beta,\lambda}_{n}(x).
\end{equation}
That is
\begin{equation}\label{eq3}
 -(\omega^{\alpha,\beta,\lambda}(x))^{-1}\frac{d}{dx}\Big\{\lambda^{-1}(1-x^{\lambda})^{\alpha+1}x^{\beta\lambda+1}\frac{d}{dx}J^{\alpha,\beta,\lambda}_{n}(x)\Big\}=\sigma^{\alpha,\beta}_{n}J^{\alpha,\beta,\lambda}_{n}(x),
\end{equation}
where $\sigma^{\alpha,\beta}_{n}=n(n+\alpha+\beta+1).$
\end{lemma}

\begin{lemma}\label{lem4}
The new defined k-th order derivatives of the fractional Jacobi polynomials are orthogonal with respect to
the weight $\omega^{\alpha+k,\beta+k,\lambda}(x)$, i.e.,
\begin{equation}\label{eq11}
\int_{0}^{1}\omega^{\alpha+k,\beta+k,\lambda}(x)D^{k}_{\lambda}J^{\alpha,\beta,\lambda}_{n}(x)D^{k}_{\lambda}J^{\alpha,\beta,\lambda}_{m}(x)dx=\widehat{h}_{n,k}^{\alpha,\beta}\delta_{m,n},
\end{equation}
where
\bq\label{peq1}
\widehat{h}^{\alpha,\beta}_{n,k}=\dps\frac{\Gamma(n+\alpha+1)\Gamma(n+\beta+1)\Gamma(n+k+\alpha+\beta+1)}{(2n+\alpha+\beta+1)(n-k)!\Gamma^{2}(n+\alpha+\beta+1)}.
\eq
Moreover, we have
\be\label{eq11b}
D^{k}_{\lambda}J^{\alpha,\beta,\lambda}_{n}(x)=\widehat{d}^{\alpha,\beta}_{n,k}J^{\alpha+k,\beta+k,\lambda}_{n-k}(x),
\ee
where
\bex 
\widehat{d}^{\alpha,\beta}_{n,k}=\dps\frac{\Gamma(n+k+\alpha+\beta+1)}{\Gamma(n+\alpha+\beta+1)}.
\eex
\end{lemma}
\begin{proof}
We start with the case $k=1$.
Using integration by parts, Lemma \ref{lemma4},
and the orthogonality of $\big\{J^{\alpha,\beta,\lambda}_{n}\big\}_{n=0}^{\infty}$, we obtain
\bq\label{peq8}
\dps\int_{0}^{1}\omega^{\alpha+1,\beta+1,\lambda}(x)D^{1}_{\lambda}J^{\alpha,\beta,\lambda}_{n}(x)D^{1}_{\lambda}J^{\alpha,\beta,\lambda}_{m}(x)dx=(J^{\alpha,\beta,\lambda}_{n},\mathcal{L}^{\alpha,\beta}_{\lambda}J^{\alpha,\beta,\lambda}_{m})_{\omega^{\alpha,\beta,\lambda}}=\sigma^{\alpha,\beta}_{n}\widehat{\gamma}^{\alpha,\beta}_{n}\delta_{m,n}.
\eq
This means that $\big\{D^{1}_{\lambda}J^{\alpha,\beta,\lambda}_{n}\big\}_{n=1}^{\infty}$
are orthogonal with respect to the weight $\omega^{\alpha+1,\beta+1,\lambda}$.
Furthermore it follows from Lemma \ref{plem1} that $D^{1}_{\lambda}J^{\alpha,\beta,\lambda}_{n}$
must be proportional to $J^{\alpha+1,\beta+1,\lambda}_{n-1},$ namely,
\bq\label{pequa4}
D^{1}_{\lambda}J^{\alpha,\beta,\lambda}_{n}(x)=\widehat{d}^{\alpha,\beta}_{n,1}J^{\alpha+1,\beta+1,\lambda}_{n-1}(x),
\eq
where $\widehat{d}^{\alpha,\beta}_{n,1}$ is a constant.\\
In virtue of \eqref{eqx5}, the leading coefficient of $J^{\alpha,\beta,\lambda}_{n}$
is $\dps k^{\alpha,\beta}_{n}=\frac{\Gamma(2n+\alpha+\beta+1)}{n!\Gamma(n+\alpha+\beta+1)}.$
Thus comparing the leading coefficients on both sides of \eqref{pequa4} gives:
\beq
\widehat{d}^{\alpha,\beta}_{n,1}=\dps\frac{nk^{\alpha,\beta}_{n}}{k^{\alpha+1,\beta+1}_{n-1}}=n+\alpha+\beta+1.
\eeq
This proves \eqref{eq11b} for $k=1$.
Applying \eqref{pequa4} recursively, we obtain
\bex
D^{k}_{\lambda}J^{\alpha,\beta,\lambda}_{n}=\widehat{d}^{\alpha,\beta}_{n,k}J^{\alpha+k,\beta+k,\lambda}_{n-k}(x),
\eex
where
\bex
\widehat{d}^{\alpha,\beta}_{n,k}=\dps\frac{\Gamma(n+k+\alpha+\beta+1)}{\Gamma(n+\alpha+\beta+1)}.
\eex
This proves \eqref{eq11b} for all $k\ge 1$. Finally,
using Lemma \ref{lem8} gives
\bex
\int_{0}^{1}\omega^{\alpha+k,\beta+k,\lambda}(x)D^{k}_{\lambda}J^{\alpha,\beta,\lambda}_{n}(x)D^{k}_{\lambda}J^{\alpha,\beta,\lambda}_{m}(x)dx=\widehat{h}_{n,k}^{\alpha,\beta}\delta_{m,n},
\eex
where
\bearr
\widehat{h}^{\alpha,\beta}_{n,k}=(\widehat{d}^{\alpha,\beta}_{n,k})^{2}\widehat{\gamma}^{\alpha+k,\beta+k}_{n-k}
=\dps\frac{\Gamma(n+\alpha+1)\Gamma(n+\beta+1)\Gamma(n+k+\alpha+\beta+1)}{(2n+\alpha+\beta+1)(n-k)!\Gamma^{2}(n+\alpha+\beta+1)}, \ \ n\geq k.
\eearr
The proof is completed.
\end{proof}

\subsection{$L^{2}_{\omega^{\alpha,\beta,\lambda}}(I)$-orthogonal projector with $\alpha,\beta>-1$}

Let  $\pi_{N,\omega^{\alpha,\beta,\lambda}}: L^{2}_{\omega^{\alpha,\beta,\lambda}}(I)\rightarrow P^{\lambda}_{N}(I)$
be the $L^{2}_{\omega^{\alpha,\beta,\lambda}}$-orthogonal projection operator defined by:
for all $v\in L^{2}_{\omega^{\alpha,\beta,\lambda}}(I)$, $\pi_{N,\omega^{\alpha,\beta,\lambda}}v\in P^{\lambda}_{N}(I)$ such that
\bex
(v-\pi_{N,\omega^{\alpha,\beta,\lambda}}v,v_{N})_{\omega^{\alpha,\beta,\lambda}}=0,~~\forall v_{N}\in P_{N}^{\lambda}(I).
\eex
Equivalently, $\pi_{N,\omega^{\alpha,\beta,\lambda}}$ can be characterized by:
\be\label{piN}
\pi_{N,\omega^{\alpha,\beta,\lambda}}v(x)=\sum_{n=0}^{N}\hat v^{\alpha,\beta}_{n}J^{\alpha,\beta,\lambda}_{n}(x),
\ee
where
$J^{\alpha,\beta,\lambda}_{n}(x)$ are the fractional Jacobi polynomials defined in \eqref{eq2}, and
\bex
\hat v^{\alpha,\beta}_{n}=\dfrac{(v,J^{\alpha,\beta,\lambda}_{n})_{\omega^{\alpha,\beta,\lambda}}}
{\|J^{\alpha,\beta,\lambda}_{n}\|^{2}_{0,\omega^{\alpha,\beta,\lambda}}}.
\eex
Immediately, we have
\bex
\|v-\pi_{N,\omega^{\alpha,\beta,\lambda}}v\|_{0,\omega^{\alpha,\beta,\lambda}}
=\inf_{v_{N}\in P^{\lambda}_{N}(I)}\|v-v_{N}\|_{0,\omega^{\alpha,\beta,\lambda}}.
\eex
To measure the truncation error, we introduce the non-uniformly fractional Jacobi-weighted Sobolev space:
\beq
B^{m,\lambda}_{\alpha,\beta}(I):=\big\{v: D^{k}_{\lambda}v\in L^{2}_{\omega^{\alpha+k,\beta+k,\lambda}}(I),0\leq k\leq m\big\},~~~~~m\in\mathbb{N},
\eeq
equipped with the inner product, norm, and semi-norm as follows:
\beq
(u,v)_{B^{m,\lambda}_{\alpha,\beta}}=\dps\sum_{k=0}^{m}(D^{k}_{\lambda}u,D^{k}_{\lambda}v)_{\omega^{\alpha+k,\beta+k,\lambda}},~~\|v\|_{B^{m,\lambda}_{\alpha,\beta}}=(v,v)^{1/2}_{B^{m,\lambda}_{\alpha,\beta}},~~~~|v|_{B^{m,\lambda}_{\alpha,\beta}}=\|D^{m}_{\lambda}v\|_{0,\omega^{\alpha+m,\beta+m,\lambda}}.
\eeq
The special case $\lambda=1$ gives the classical non-uniformly Jacobi-weighted Sobolev space:
\beq
B^{m,1}_{\alpha,\beta}(I):=\big\{v: \partial^{k}_{x}v\in L^{2}_{\omega^{\alpha+k,\beta+k,1}}(I),0\leq k\leq m\big\},~~~~~m\in\mathbb{N}
\eeq
The functions in $B^{m,\lambda}_{\alpha,\beta}(I)$ and the ones in $B^{m,1}_{\alpha,\beta}(I)$
are linked through the following lemma.
\begin{lemma}\label{plem2}
A function $v$ belongs to $B^{m,\lambda}_{\alpha,\beta}(I)$ if and only if
$v(x^{1/\lambda})$ belongs to $B^{m,1}_{\alpha,\beta}(I).$
\end{lemma}
\begin{proof}
Using the variable change $x=t^{1/\lambda}$, we have
\barr\label{peq6}
\|D^{k}_{\lambda}v\|^{2}_{0,\omega^{\alpha+k,\beta+k,\lambda}}&=\dps\int_{0}^{1}\lambda(1-x^{\lambda})^{\alpha+k}x^{(\beta+k+1)\lambda-1}\big(D^{k}_{\lambda}v\big)^{2}dx\\[9pt]
&=\dps\int_{0}^{1}(1-t)^{\alpha+k}t^{\beta+k}\big(\partial^{k}_{t}\big\{v(t^{1/\lambda})\big\}\big)^{2}dt\\[9pt]
&=\|\partial^{k}_{x}\big\{v(x^{1/\lambda})\big\}\|^{2}_{0,\omega^{\alpha+k,\beta+k,1}},~~~~0\leq k\leq m.
\earr
This proves the desired result.
\end{proof}
\begin{proposition}\label{th1}
The orthogonal projector $\pi_{N,\omega^{\alpha,\beta,\lambda}}$ admits
the following error estimate:
for any $ v(x^{\frac{1}{\lambda}})\in B^{m,1}_{\alpha,\beta}(I),$ and $0\leq l\leq m\leq N+1,$
\bq\label{peq4}
\|D^{l}_{\lambda}(v-\pi_{N,\omega^{\alpha,\beta,\lambda}}v)\|_{0,\omega^{\alpha+l,\beta+l,\lambda}}\leq\dps c\sqrt{\frac{(N-m+1)!}{(N-l+1)!}}N^{(l-m)/2}\|\partial^{m}_{x}\big\{v(x^{\frac{1}{\lambda}})\big\}\|_{0,\omega^{\alpha+m,\beta+m,1}}.
\eq
For a fixed $m$, the above estimate can be simplified as
\bq\label{peq5}
\|D^{l}_{\lambda}(v-\pi_{N,\omega^{\alpha,\beta,\lambda}}v)\|_{0,\omega^{\alpha+l,\beta+l,\lambda}}\leq\dps cN^{l-m}\|\partial^{m}_{x}\big\{v(x^{\frac{1}{\lambda}})\big\}\|_{0,\omega^{\alpha+m,\beta+m,1}},
\eq
where $c\approx 1$ for $N\gg1.$
 In particular, for $l=0, 1$, we have
 \bq\label{pequa5}
 \|v-\pi_{N,\omega^{\alpha,\beta,\lambda}}v\|_{0,\omega^{\alpha,\beta,\lambda}}\leq\dps cN^{-m}\|\partial^{m}_{x}\big\{v(x^{\frac{1}{\lambda}})\big\}\|_{0,\omega^{\alpha+m,\beta+m,1}},
 \eq
\begin{equation}\label{pequ1}
\|\partial_{x}(v-\pi_{N,\omega^{\alpha,\beta,\lambda}}v)\|_{0,\tilde{\omega}^{\alpha,\beta,\lambda}}\leq cN^{1-m}\|\partial_{x}^{m}\big\{v(x^{\frac{1}{\lambda}})\big\}\|_{0,\omega^{\alpha+m,\beta+m,1}},
\end{equation}
where $\tilde{\omega}^{\alpha,\beta,\lambda}=\lambda^{-1}(1-x^{\lambda})^{\alpha+1}x^{\beta\lambda+1}$.
\end{proposition}
\begin{proof}
For $v(x^{\frac{1}{\lambda}})\in B^{m,1}_{\alpha,\beta}(I)$,
Lemma \ref{plem2} gives $ v(x)\in B^{m,\lambda}_{\alpha,\beta}(I)$.
Thanks to \eqref{piN} and the orthogonality \eqref{eq11}, we have
\bex
\|D^{l}_{\lambda}(v-\pi_{N,\omega^{\alpha,\beta,\lambda}}v)\|^{2}_{0,\omega^{\alpha+l,\beta+l,\lambda}}
=\dps\sum_{n=N+1}^{\infty}\hat{h}^{\alpha,\beta}_{n,l}|\hat{v}^{\alpha,\beta}_{n}|^{2}
=\sum_{n=N+1}^{\infty}\frac{\hat{h}^{\alpha,\beta}_{n,l}}{\hat{h}^{\alpha,\beta}_{n,m}}\hat{h}^{\alpha,\beta}_{n,m}|\hat{v}^{\alpha,\beta}_{n}|^{2}.
\eex
Further estimation on the right hand side term gives
\barr\label{peq2}
\|D^{l}_{\lambda}(v-\pi_{N,\omega^{\alpha,\beta,\lambda}}v)\|^{2}_{0,\omega^{\alpha+l,\beta+l,\lambda}}
\dps\leq\max_{n\geq N+1}\Big\{\frac{\hat{h}^{\alpha,\beta}_{n,l}}{\hat{h}^{\alpha,\beta}_{n,m}}\Big\}\sum_{n=0}^{\infty}\hat{h}^{\alpha,\beta}_{n,m}|\hat{v}^{\alpha,\beta}_{n}|^{2}
\leq \frac{\hat{h}^{\alpha,\beta}_{N+1,l}}{\hat{h}^{\alpha,\beta}_{N+1,m}}\|D^{m}_{\lambda}v\|^{{\color{black}2}}_{0,\omega^{\alpha+m,\beta+m,\lambda}}.
\earr
In virtue of \eqref{peq1}, we find that for $0\leq l\leq m\leq N+1$,
\barr\label{peq3}
\dps\frac{\hat{h}^{\alpha,\beta}_{N+1,l}}{\hat{h}^{\alpha,\beta}_{N+1,m}}&=\dps\frac{\Gamma(N+\alpha+\beta+l+2)}{\Gamma(N+\alpha+\beta+m+2)}\frac{(N-m+1)!}{(N-l+1)!}\\[9pt]
&=\dps\frac{1}{(N+\alpha+\beta+l+2)(N+\alpha+\beta+l+3)\cdots(N+\alpha+\beta+1+m)}\frac{(N-m+1)!}{(N-l+1)!}\\[9pt]
&\dps\leq N^{l-m}\frac{(N-m+1)!}{(N-l+1)!},
\earr
where we have used the facts: $\alpha+\beta+2>0.$ Then the estimate \eqref{peq4} follows from \eqref{peq2}, \eqref{peq3}, and \eqref{peq6}.
\\
Next, we prove \eqref{peq5}. Recall the property of the Gamma function (see \cite[(6.1.38)]{Abr72}):
\beq
\Gamma(x+1)=\sqrt{2\pi}x^{x+1/2}\exp\big(-x+\frac{\theta}{12x}\big),\ \ \ \forall x>0,\ 0<\theta<1.
\eeq
Moreover it can be shown that for any constant $a,b\in \mathbb{R}, n+a>1$ and $ n+b>1$
(see \cite[Lemma 2.1]{Zhao13},\cite{CSW14}),
\beq
\dps\frac{\Gamma(n+a)}{\Gamma(n+b)}\leq\nu^{a,b}_{n}n^{a-b},
\eeq
where
\beq
\nu^{a,b}_{n}=\exp\Big(\frac{a-b}{2(n+b-1)}+\frac{1}{12(n+a-1)}+{\color{black}\frac{(a-1)(b-1)}{n}}\Big).
\eeq
Hence, we obtain
\bq\label{peq7}
\dps\frac{(N-m+1)!}{(N-l+1)!}=\frac{\Gamma(N-m+2)}{\Gamma(N-l+2)}\leq \nu^{2-m,2-l}_{N}N^{l-m},
\eq
where $\nu^{2-m,2-l}_{N}\approx1$ for fixed $m$ and $N\gg1.$ \\
Combining \eqref{peq4} and \eqref{peq7} gives \eqref{peq5}.
Furthermore, using \eqref{peq5} and the definition of $D^{1}_{\lambda}$ we obtain
\eqref{pequa5} and \eqref{pequ1}.
\end{proof}

\subsection{$I^{\alpha,\beta}_{N,\lambda}$-interpolation on fractional Jacobi-Gauss-type points with $\alpha,\beta>-1$}

Let $h^{\alpha,\beta}_{j,\lambda}(x)$ be the generalized Lagrange basis function:
\begin{equation}\label{eq4}
h^{\alpha,\beta}_{j,\lambda}(x)=\prod_{i=0,i\neq j}^{N}\frac{x^{\lambda}-x_{i}^{\lambda}}{x^{\lambda}_{j}-x^{\lambda}_{i}},~~~~0\leq j\leq N,
\end{equation}
where $x_{0}<x_{1}< \dots < x_{N-1}<x_{N}$ are zeros in $I$ of $J^{\alpha,\beta,\lambda}_{N+1}(x)$.
It is clear that the functions $h^{\alpha,\beta}_{j,\lambda}(x)$ satisfy
\begin{equation*}
h^{\alpha,\beta}_{j,\lambda}(x_{i})=\delta_{ij}.
\end{equation*}
Let $z(x)=x^{\lambda}$. Then $z_{i}:=z(x_{i})=x_{i}^{\lambda},~~0\leq i\leq N$,
are zeros of $J^{\alpha,\beta,1}_{N+1}(x)$, and
\begin{equation}\label{equ6}
h^{\alpha,\beta}_{j,\lambda}(x)
=h^{\alpha,\beta}_{j,1}(z)
:=\prod_{i=0,i\neq j}^{N}\frac{z-z_{i}}{z_{j}-z_{i}},~~~~0\leq j\leq N.
\end{equation}
We define the generalized interpolation operator $I^{\alpha,\beta}_{N,\lambda}$ by
\begin{equation*}
I^{\alpha,\beta}_{N,\lambda}v(x)=\sum_{j=0}^{N}v(x_{j})h^{\alpha,\beta}_{j,\lambda}(x).
\end{equation*}
Then
\begin{equation}\label{lemm3}
I^{\alpha,\beta}_{N,\lambda}v(x)
=\dps\sum_{j=0}^{N}v(z_{j}^{1/\lambda})h^{\alpha,\beta}_{j,1}(z)=\dps I^{\alpha,\beta}_{N,1}v(z^{1/\lambda}), \ \ z=x^{\lambda}.
\end{equation}

\begin{lemma}[Case $\lambda=1$, \cite{STW10} Lemma 3.8]\label{lem1}
The interpolation operator $I^{\alpha,\beta}_{N,1}$ satisfies following stability estimate:
for any $v\in B^{1,1}_{\alpha,\beta}(I)$,
\begin{equation*}
\|I^{\alpha,\beta}_{N,1}v\|_{0,\omega^{\alpha,\beta,1}}\leq c\big(\|v\|_{0,\omega^{\alpha,\beta,1}}
+N^{-1}\|\partial_{x}v\|_{0,\omega^{\alpha+1,\beta+1,1}}\big).
\end{equation*}
\end{lemma}
\begin{proposition}\label{th2}
For any $v(x^{\frac{1}{\lambda}})\in {\color{black}B^{1,1}_{\alpha,\beta}(I)}$, we have
\begin{equation}\label{pequ6}
\|I^{\alpha,\beta}_{N,\lambda}v\|_{0,\omega^{\alpha,\beta,\lambda}}
\leq c\big(\|v\|_{0,\omega^{\alpha,\beta,\lambda}}
+N^{-1}\|D^{1}_{\lambda}v\|_{0,\omega^{\alpha+1,\beta+1,\lambda}}\big).
\end{equation}
\end{proposition}
\begin{proof}
Making the variable change $z=x^{\lambda}$ in the definition of $I^{\alpha,\beta}_{N,\lambda}v(x)$,
we get
\begin{equation}\label{equati1}
\begin{array}{r@{}l}
\|I^{\alpha,\beta}_{N,\lambda}v\|^{2}_{{\color{black}0,}\omega^{\alpha,\beta,\lambda}}&
=\dps\int_{0}^{1}\Big[\sum^{N}_{i=0}v(x_{i})h_{i,1}^{\alpha,\beta}(x^{\lambda})\Big]^{2}\lambda(1-x^{\lambda})^{\alpha} x^{(\beta+1)\lambda-1}dx\\[9pt]
&=\dps\int_{0}^{1}\Big[\sum^{N}_{i=0}v(z^{1/\lambda}_{i})h_{i,1}^{\alpha,\beta}(x^{\lambda})\Big]^{2}\lambda(1-x^{\lambda})^{\alpha} x^{(\beta+1)\lambda-1}dx\\[9pt]
&=\dps\int_{0}^{1}\Big[\sum^{N}_{i=0}v(z^{1/\lambda}_{i})h_{i,1}^{\alpha,\beta}(z)\Big]^{2}(1-z)^{\alpha} z^{\beta}dz\\[9pt]
&=\|I^{\alpha,\beta}_{N,1}v(z^{1/\lambda})\|^{2}_{{\color{black}0,}\omega^{\alpha,\beta,1}}.\\
\end{array}
\end{equation}
Then it follows from Lemma \ref{lem1} and \eqref{equati1}
\begin{equation}\label{pequ5}
\begin{array}{r@{}l}
\|I^{\alpha,\beta}_{N,\lambda}v(x)\|_{{\color{black}0,}\omega^{\alpha,\beta,\lambda}}
=\|I^{\alpha,\beta}_{N,1}v(x^{1/\lambda})\|_{{\color{black}0,}\omega^{\alpha,\beta,1}}
\leq c(\|v(x^{\frac{1}{\lambda}})\|_{0,\omega^{\alpha,\beta,1}}+N^{-1}\|\partial_{x}v(x^{\frac{1}{\lambda}})\|_{0,\omega^{\alpha+1,\beta+1,1}}).\\
\end{array}
\end{equation}
The desired result \eqref{pequ6} follows from
combining \eqref{pequ5} and \eqref{peq6}.
\end{proof}
We will also need some inverse inequalities for the $\lambda-$polynomials, which are given in the following
proposition.
\begin{proposition}
For any $\phi\in P^{\lambda}_{N}(I)$, we have
\begin{equation}\label{equ5}
\|\partial_{x}\phi\|_{0,\tilde{\omega}^{\alpha,\beta,\lambda}}\leq \sqrt{\sigma^{\alpha,\beta}_{N}}\|\phi\|_{0,\omega^{\alpha,\beta,\lambda}},
\end{equation}
and
\bq\label{pequ4}
\|D^{k}_{\lambda}\phi\|_{0,\omega^{\alpha+k,\beta+k,\lambda}}\leq c N^{k}\|\phi\|_{0,\omega^{\alpha,\beta,\lambda}},\ \ \ k\geq1,
\eq
where $\sigma^{\alpha,\beta}_{N}$ is defined in \eqref{eq3} and $c\approx1$ for fixed $k$ and $N\gg1$.
\end{proposition}
\begin{proof}
For any $\phi\in P^{\lambda}_{N}(I)$, we express it under the form
\begin{equation}\label{equ2}
\phi(x)=\sum_{i=0}^{N}\hat{\phi}_{i}^{\alpha,\beta}J^{\alpha,\beta,\lambda}_{i}(x), \mbox{ with }\dps\hat{\phi}_{i}^{\alpha,\beta}=\frac{(\phi,J^{\alpha,\beta,\lambda}_{i})_{\omega^{\alpha,\beta,\lambda}}}
{\hat{\gamma}_{i}^{\alpha,\beta}},
\end{equation}
where $\hat{\gamma}_{i}^{\alpha,\beta}$ was defined in \eqref{eq12}.
Then we deduce from the orthogonality of $\{J^{\alpha,\beta,\lambda}_{i}(x)\}_{i=0}^{N}$
\begin{equation*}
\|\phi\|^{2}_{0,\omega^{\alpha,\beta,\lambda}}=\sum_{i=0}^{N}\hat{\gamma}_{i}^{\alpha,\beta}|\hat{\phi}_{i}^{\alpha,\beta}|^{2}.
\end{equation*}
Applying the differential operator $D^{1}_{\lambda}$ to the both sides of \eqref{equ2}, and using the orthogonality \eqref{peq8}, we get
\begin{equation*}
\|D^{1}_{\lambda}\phi\|^{2}_{0,\omega^{\alpha+1,\beta+1,\lambda}}=\|\partial_{x}\phi\|^{2}_{0,\tilde{\omega}^{\alpha,\beta,\lambda}}=\dps\sum_{i=0}^{N}\sigma^{\alpha,\beta}_{i}\hat{\gamma}_{i}^{\alpha,\beta}|\hat{\phi}_{i}^{\alpha,\beta}|^{2}
\leq\sigma^{\alpha,\beta}_{N}\dps\sum_{i=0}^{N}\hat{\gamma}_{i}^{\alpha,\beta}|\hat{\phi}_{i}^{\alpha,\beta}|^{2}
\leq \sigma^{\alpha,\beta}_{N}\|\phi\|_{0,\omega^{\alpha,\beta,\lambda}}^{2}.
\end{equation*}
This results in \eqref{equ5}.
\\
Applying the operators $D^{1}_{\lambda}$ $k$ times to \eqref{equ2} and using \eqref{eq11},
we obtain
\barr\label{pequ2}
\|D^{k}_{\lambda}\phi\|^{2}_{0,\omega^{\alpha+k,\beta+k,\lambda}}&=\dps\sum_{n=k}^{N}\hat{h}^{\alpha,\beta}_{n,k}|\hat{\phi}^{\alpha,\beta}_{n}|^{2}=\sum_{n=k}^{N}\frac{\Gamma(n+k+\alpha+\beta+1)n!}{\Gamma(n+\alpha+\beta+1)(n-k)!}\hat{\gamma}^{\alpha,\beta}_{n}|\hat{\phi}^{\alpha,\beta}_{n}|^{2}\\[9pt]
&\leq\dps\frac{\Gamma(N+k+\alpha+\beta+1)N!}{\Gamma(N+\alpha+\beta+1)(N-k)!}\sum_{n=k}^{N}\hat{\gamma}^{\alpha,\beta}_{n}|\hat{\phi}^{\alpha,\beta}_{n}|^{2}\\[9pt]
&\leq\dps\frac{\Gamma(N+k+\alpha+\beta+1)N!}{\Gamma(N+\alpha+\beta+1)(N-k)!}\|\phi\|^{2}_{0,\omega^{\alpha,\beta,\lambda}}.\\
\earr
Similar to the discussion for \eqref{peq7}, we have
\bq\label{pequ3}
\dps\frac{\Gamma(N+k+\alpha+\beta+1)N!}{\Gamma(N+\alpha+\beta+1)(N-k)!}\leq \nu^{k+\alpha+\beta+1,\alpha+\beta+1}_{N}\nu^{1,1-k}_{N}N^{2k},
\eq
where $\nu^{k+\alpha+\beta+1,\alpha+\beta+1}_{N}\nu^{1,1-k}_{N}\approx1$ for fixed $k$ and $N\gg1.$
\\
Then the inequality \eqref{pequ4} follows from combining \eqref{pequ2} and \eqref{pequ3}.
\end{proof}

We now present the error estimation for the fractional Jacobi-Gauss interpolation operator
based on the roots of $J^{\alpha,\beta,\lambda}_{N+1}(x)$.
\begin{proposition}\label{th3}
It holds for any $v(x^{1/\lambda})\in B^{m,1}_{\alpha,\beta}(I), m\geq1$, and any $0\leq l\leq m\leq N+1$,
\begin{equation*}
\|D^{l}_{\lambda}\big(v-I^{\alpha,\beta}_{N,\lambda}v\big)\|_{0,\omega^{\alpha+l,\beta+l,\lambda}}\dps\leq c\sqrt{\frac{(N-m+1)!}{N!}}N^{l-(m+1)/2}\|\partial_{x}^{m}\big\{v(x^{\frac{1}{\lambda}})\big\}\|_{0,\omega^{\alpha+m,\beta+m,1}}.
\end{equation*}
If $m$ is fixed, then the above estimate can be simplified as
\begin{equation}\label{pequa1}
\|D^{l}_{\lambda}\big(v-I^{\alpha,\beta}_{N,\lambda}v\big)\|_{0,\omega^{\alpha+l,\beta+l,\lambda}}\dps\leq cN^{l-m}\|\partial_{x}^{m}\big\{v(x^{\frac{1}{\lambda}})\big\}\|_{0,\omega^{\alpha+m,\beta+m,1}},
\end{equation}
where $c\approx1$ for $N\gg1.$
 In particular, for $l=0, 1$ we have
\bq\label{pequa6}
\|v-I^{\alpha,\beta}_{N,\lambda}v\|_{0,\omega^{\alpha,\beta,\lambda}}\leq cN^{-m}\|\partial_{x}^{m}\big\{v(x^{\frac{1}{\lambda}})\big\}\|_{0,\omega^{\alpha+m,\beta+m,1}},
\eq
\begin{equation}\label{pequa2}
\|\partial_{x}(v-I^{\alpha,\beta}_{N,\lambda}v)\|_{0,\tilde{\omega}^{\alpha,\beta,\lambda}}\leq cN^{1-m}\|\partial_{x}^{m}\big\{v(x^{\frac{1}{\lambda}})\big\}\|_{0,\omega^{\alpha+m,\beta+m,1}}.
\end{equation}
\end{proposition}
\begin{proof}
From \eqref{peq4} and Proposition \ref{th2}, we deduce
\begin{equation}\label{equ4}
\begin{array}{r@{}l}
&\|I^{\alpha,\beta}_{N,\lambda}v-\pi_{N,\omega^{\alpha,\beta,\lambda}}v\|_{0,\omega^{\alpha,\beta,\lambda}}
=\|I^{\alpha,\beta}_{N,\lambda}(v-\pi_{N,\omega^{\alpha,\beta,\lambda}}v)\|_{0,\omega^{\alpha,\beta,\lambda}}\\[9pt]
&\ \ \ \leq c\big(\|v(x)-\pi_{N,\omega^{\alpha,\beta,\lambda}}v(x)\|_{0,\omega^{\alpha,\beta,\lambda}}+N^{-1}\|D^{1}_{\lambda}\big\{v(x)-\pi_{N,\omega^{\alpha,\beta,\lambda}}v(x)\big\}\|_{0,\omega^{\alpha+1,\beta+1,\lambda}}\big)\\[9pt]
&\ \ \ \dps\leq c\sqrt{\frac{(N-m+1)!}{N!}}N^{-(m+1)/2}\|\partial_{x}^{m}\big\{v(x^{\frac{1}{\lambda}})\big\}\|_{0,\omega^{\alpha+m,\beta+m,1}}.
\end{array}
\end{equation}
Furthermore, using the inverse inequality {\color{black}\eqref{pequ4}}, we obtain
\begin{equation*}
\begin{array}{r@{}l}
&\|D^{l}_{\lambda}\big(I^{\alpha,\beta}_{N,\lambda}v-\pi_{N,\omega^{\alpha,\beta,\lambda}}v\big)\|_{0,\omega^{\alpha+l,\beta+l,\lambda}}\leq cN^{l}\|I^{\alpha,\beta}_{N,\lambda}v-\pi_{N,\omega^{\alpha,\beta,\lambda}}v\|_{0,\omega^{\alpha,\beta,\lambda}}\\[9pt]
&\ \ \ \ \ \dps\leq c\sqrt{\frac{(N-m+1)!}{N!}}N^{l-(m+1)/2}\|\partial_{x}^{m}\big\{v(x^{\frac{1}{\lambda}})\big\}\|_{0,\omega^{\alpha+m,\beta+m,1}}.\\
\end{array}
\end{equation*}
The above inequality, together with the triangle inequality and \eqref{peq4} gives
\begin{equation*}
\begin{array}{r@{}l}
\|D^{l}_{\lambda}\big(v-I^{\alpha,\beta}_{N,\lambda}v\big)\|_{0,\omega^{\alpha+l,\beta+l,\lambda}}&\leq \|D^{l}_{\lambda}\big(v-\pi_{N,\omega^{\alpha,\beta,\lambda}}v\big)\|_{0,\omega^{\alpha+l,\beta+l,\lambda}}\\[9pt]
&\ \ \ \ \ \ \ \ \ \ \ \ \ \ \ \ \ \ +\|D^{l}_{\lambda}\big(\pi_{N,\omega^{\alpha,\beta,\lambda}}u-I^{\alpha,\beta}_{N,\lambda}v\big)\|_{0,\omega^{\alpha+l,\beta+l,\lambda}}\\[9pt]
&\dps\leq c\sqrt{\frac{(N-m+1)!}{N!}}N^{l-(m+1)/2}\|\partial_{x}^{m}\big\{v(x^{\frac{1}{\lambda}})\big\}\|_{0,\omega^{\alpha+m,\beta+m,1}}.
\end{array}
\end{equation*}
This proves the first estimate.
In case $m$ is fixed, the estimates \eqref{pequa1}-\eqref{pequa2} can be derived in a similar way as
in Proposition \ref{th1}.
The proof is completed.
\end{proof}
\begin{lemma}[Weighted Sobolev inequality]\label{lem6}
For all $v(x)\in B^{1,\lambda}_{\alpha,\beta}(I), v(\xi)=0$ for some $\xi\in[0,1]$, and
any $-1<\alpha,\beta<-\frac{1}{2}$, it holds
\begin{equation}\label{equatio6}
\|v\|_{\infty}\leq \sqrt{2}\|v\|^{1/2}_{0,\omega^{\alpha,\beta,\lambda}}\|\partial_{x}v\|^{1/2}_{0,\tilde{\omega}^{\alpha,\beta,\lambda}}.
\end{equation}
\end{lemma}
\begin{proof}
Using the Cauchy-Schwarz inequality, we have, for all $x\in [\xi, 1]$,
\begin{equation}\label{equatio2}
\begin{array}{r@{}l}
v^{2}(x)&=\dps\int_{\xi}^{x}dv^{2}(x)=\int_{\xi}^{x}2v(x)\partial_{x} v(x)dx\leq 2\int_{\xi}^{1}|v\partial_{x} v|dx\\
&\dps=2\int_{\xi}^{1}\lambda^{1/2}(1-x^{\lambda})^{\frac{\alpha}{2}}x^{\frac{(\beta+1)\lambda-1}{2}}|v|\Big[\lambda^{-1/2}(1-x^{\lambda})^{-\frac{\alpha}{2}}x^{-\frac{(\beta+1)\lambda-1}{2}}| \partial_{x}v|\Big]dx\\
&\dps\leq 2\|v\|_{0,\omega^{\alpha,\beta,\lambda}}\Big[\int_{0}^{1}\lambda^{-1}(1-x^{\lambda})^{-\alpha}x^{-(\beta+1)\lambda+1}(\partial_{x}v)^{2}dx\Big]^{1/2}.
\end{array}
\end{equation}
The condition $-1<\alpha,\beta<-\dfrac{1}{2}$ guarantees that it holds for any $x\in I$,
\begin{equation}\label{equatio3}
\lambda^{-1}(1-x^{\lambda})^{-\alpha}x^{-(\beta+1)\lambda+1}\leq \lambda^{-1}(1-x^{\lambda})^{\alpha+1}x^{\beta\lambda+1}=\tilde{\omega}^{\alpha,\beta,\lambda}(x).
\end{equation}
Bringing \eqref{equatio3} into \eqref{equatio2} gives
\begin{equation}\label{equatio4}
\dps\max_{x\in[\xi,1]}|v|\leq \sqrt{2}\|v\|_{0,\omega^{\alpha,\beta,\lambda}}^{1/2}\|\partial_{x}v\|_{0,\tilde{\omega}^{\alpha,\beta,\lambda}}^{1/2}.
\end{equation}
In a similar way we can prove
\begin{equation}\label{equatio5}
\dps\max_{x\in[0,\xi]}|v|\leq \sqrt{2}\|v\|_{0,\omega^{\alpha,\beta,\lambda}}^{1/2}\|\partial_{x}v\|_{0,\tilde{\omega}^{\alpha,\beta,\lambda}}^{1/2}.
\end{equation}
This completes the proof.
\end{proof}

\begin{proposition}[Interpolation error in $L^\infty-$norm]\label{th4}
If $-1<\alpha, \beta\leq-\dfrac{1}{2},$ we have
\begin{equation*}
\|v-I_{N,\lambda}^{\alpha,\beta}v\|_{\infty}\leq c N^{1/2-m}\|\partial^{m}_{x}v(x^{1/\lambda})\|_{0,\omega^{\alpha+m,\beta+m,1}},~~\forall v(x^{1/\lambda})\in B^{m,1}_{\alpha,\beta}(I), \ m\ge 1.
\end{equation*}
\end{proposition}
\begin{proof}
Obviously if $v(x^{1/\lambda})\in B^{m,1}_{\alpha,\beta}(I), m\ge 1$,
then $v(x)\in B^{1,\lambda}_{\alpha,\beta}(I)$.
By definition, the function $v(x)-I^{\alpha,\beta}_{N,\lambda}v(x)$ vanishes at all the roots of $J^{\alpha,\beta,\lambda}_{N+1}(x)$. Thus we can use
Lemma \ref{lem6} and Proposition \ref{th3} to conclude
\bry
\|v-I_{N,\lambda}^{\alpha,\beta}v\|_{\infty}
&\leq\sqrt{2}\|v-I_{N,\lambda}^{\alpha,\beta}v\|_{0,\omega^{\alpha,\beta,\lambda}}^{1/2}\|\partial_{x}(v-I_{N,\lambda}^{\alpha,\beta,}v)\|_{0,\tilde{\omega}^{\alpha,\beta,\lambda}}^{1/2}\\[9pt]
&\leq c N^{1/2-m}\|\partial^{m}_{x}v(x^{1/\lambda})\|_{0,\omega^{\alpha+m,\beta+m,1}}.
\ery
This proves the proposition.
\end{proof}

We use the Jacobi Gauss points to define the discrete inner product: for any $u,v\in C(\bar I)$,
\beq
(u,v)_{N,\omega^{\alpha,\beta,1}}=\sum_{i=0}^{N}{\color{black}u}(z_{i})v(z_{i})\omega_{i},
\eeq
where $\{z_{i}\}_{i=0}^{N}$ are the zeros of the shifted Jacobi polynomial $J^{\alpha,\beta,1}_{N+1}(x)$ and $\{\omega_{i}\}_{i=0}^{N}$ are the corresponding weights.

\begin{lemma}[\cite{STW10,Ber97}]\label{lem3}
For all $v\in B^{m,1}_{\alpha,\beta}(I), m\geq 1$ and all $\phi\in P^{1}_{N}(I)$, we have
\begin{equation*}
|(v,\phi)_{\omega^{\alpha,\beta,1}}-(v,\phi)_{N,\omega^{\alpha,\beta,1}}|\leq cN^{-m}\|\partial_{x}^{m}v\|_{0,\omega^{\alpha+m,\beta+m,1}}\|\phi\|_{0,\omega^{\alpha,\beta,1}}.
\end{equation*}
\end{lemma}
The following result, which can be found in \cite{Mas01},
concerns the Lebesgue constant of the Lagrange interpolation polynomials
associated with the zeros of shifted Jacobi polynomials.
\begin{lemma}\label{lem2}
Let $\{h^{\alpha,\beta}_{j,1}(x)\}_{j=0}^{N}$ be the Lagrange interpolation polynomials associated with the Gauss points of the shifted Jacobi polynomial $J^{\alpha,\beta,1}_{N+1}(x).$ Then
\begin{equation*}
\|I^{\alpha,\beta}_{N,1}\|_{\infty}:=\max_{x\in I}\sum_{j=0}^{N}|h_{j,1}^{\alpha,\beta}(x)|=
\begin{cases}
O(\log N),~~~-1<\alpha,\beta\leq -\frac{1}{2},\\
O(N^{\gamma+\frac{1}{2}}),~~~\gamma=\max(\alpha,\beta), otherwise.\\
\end{cases}
\end{equation*}
\end{lemma}
The above classical results can be readily extended to the generalized Lagrange interpolation polynomials
associated with the zeros of fractional Jacobi polynomials, which is given in the following lemma.
\begin{lemma}\label{lemm1}
Let $\{h^{\alpha,\beta}_{j,\lambda}(x)\}_{j=0}^{N}$ be the generalized Lagrange interpolation basis functions associated with the Gauss points of the fractional Jacobi polynomial $J^{\alpha,\beta,\lambda}_{N+1}(x).$ Then
\begin{equation*}
\|I^{\alpha,\beta}_{N,\lambda}\|_{\infty}
:=\max_{x\in I}\sum_{j=0}^{N}|h_{j,\lambda}^{\alpha,\beta}(x)|
=
\begin{cases}
O(\log N),~~~-1<\alpha,\beta\leq -\frac{1}{2},\\
O(N^{\gamma+\frac{1}{2}}),~~~\gamma=\max(\alpha,\beta), otherwise.\\
\end{cases}
\end{equation*}
\end{lemma}\par
From now on, for $r\geq0$ and  $\kappa\in [0,1],$ $C^{r,\kappa}(I)$ will denote the space of functions whose
$r-$th derivatives are H\"{o}lder continuous with exponent $\kappa$, endowed with the usual norm:
\begin{equation*}
\|v\|_{r,\kappa}=\max_{0\leq i\leq r}\max_{x\in I}|\partial_{x}^{i}v(x)|+\max_{0\leq i\leq r}\sup_{x,y\in I,x\neq y}\frac{| \partial_{x}^{i}v(x)-\partial_{x}^{i}v(y)|}{|x-y|^{\kappa}}.
\end{equation*}
When $\kappa=0, C^{r,0}(I)$ denotes the space of functions with $r$ continuous derivatives on $I$,
which is also commonly denoted by $C^{r}(I)$, endowed with the norm $\|\cdot\|_{r}$.\par
  A well-known result from Ragozin \cite{Rag70,Rag71} will be useful, which states that,
  for non-negative integer $r$ and real number $\kappa\in (0,1)$, there exists a linear operator {\color{black}$\mathcal{T}_{N}$}
  from $C^{r,\kappa}(I)$ into $P_{N}^{1}(I)$,
such that
\begin{equation}\label{equat2}
\|v-\mathcal{T}_{N}v\|_{\infty}\leq C_{r,\kappa}N^{-(r+\kappa)}\|v\|_{r,\kappa}, \ \ v\in C^{r,\kappa}(I),
\end{equation}
where $C_{r,\kappa}$ is a constant which may depend on $r$ and $\kappa$.

We further define the linear, weakly singular integral operator $\mathcal{K}$:
\begin{equation}\label{equ8}
\big(\mathcal{K}v\big)(x)=\int_{0}^{x}(x-s)^{-\mu}K(x,s)v(s)ds,
\end{equation}
where $K\in C(I\times I)$ with $K(x,x)\neq 0$ for $x\in I$.
We will prove that $\mathcal{K}$ is a compact operator from $C(I)$ to $C^{0,\kappa}(I)$ for any $0<\kappa<1-\mu.$

\begin{lemma}
For any function $v\in C(I)$, {\color{black}$K\in C(I\times I)$, and $K(\cdot,s)\in C^{0,\kappa}(I)$ with} $0<\kappa<1-\mu$, we have
\begin{equation}\label{equ7}
\frac{|\big(\mathcal{K}v\big)(x)-\big(\mathcal{K}v\big)(y)|}{|x-y|^{\kappa}}\leq c \max_{x\in I}|v(x)|, ~~~\forall x,y\in I,~x\neq y.
\end{equation}
This implies that
\begin{equation*}
\|\mathcal{K}v\|_{0,\kappa}\leq c\|v\|_{\infty},~~~0<\kappa<1-\mu.
\end{equation*}
\end{lemma}
\begin{proof}
Without loss of generality, assume $0\leq y< x\leq1$. We have
\bearr
\dps\frac{|\big(\mathcal{K}v\big)(x)-\big(\mathcal{K}v\big)(y)|}{|x-y|^{\kappa}}
&\dps=(x-y)^{-\kappa}\Big |\int_{0}^{y}(y-s)^{-\mu}K(y,s)v(s)ds-\int_{0}^{x}(x-s)^{-\mu}K(x,s)v(s){\color{black}ds}\Big |\\[9pt]
&\leq M_{1}+M_{2},\\
\eearr
where
\bearr
&M_{1}=\dps(x-y)^{-\kappa}\int_{0}^{y}|(y-s)^{-\mu}K(y,s)-(x-s)^{-\mu}K(x,s)| |v(s)|ds,\\[9pt]
&M_{2}=\dps(x-y)^{-\kappa}\int_{y}^{x}(x-s)^{-\mu}|K(x,s)||v(s)|ds.\\
\eearr
For $M_1$, we use the triangle inequality to deduce
\bq\label{ppeq2}
M_{1}\leq M^{(1)}+M^{(2)},
\eq
where
\bearr
&M^{(1)}\dps=(x-y)^{-\kappa}\int_{0}^{y}|(y-s)^{-\mu}-(x-s)^{-\mu}| |K(y,s)||v(s)|ds,\\[9pt]
&M^{(2)}\dps=(x-y)^{-\kappa}\int_{0}^{y}(x-s)^{-\mu}|K(y,s)-K(x,s)||v(s)|ds.\\
\eearr
$M^{(1)}$ and $M^{(2)}$ can be bounded respectively by
\bearr
M^{(1)}
&\dps\leq c\|v\|_{\infty}(x-y)^{-\kappa}\Big[\int_{0}^{y}(y-s)^{-\mu}ds-\int_{0}^{x}(x-s)^{-\mu}ds+\int_{y}^{x}(x-s)^{-\mu}ds\Big]\\[9pt]
&\dps\leq c\|v\|_{\infty}(x-y)^{-\kappa}\Big[y^{1-\mu}\int_{0}^{1}(1-\tau)^{-\mu}d\tau
-x^{1-\mu}\int_{0}^{1}(1-\tau)^{-\mu}
d\tau\\[9pt]
&\dps\ \ \ \ +(x-y)^{1-\mu}\int_{0}^{1}(1-\tau)^{-\mu}d\tau\Big]\\[9pt]
&\leq c B(1-\mu,1)(x-y)^{1-\mu-\kappa}\|v\|_{\infty}\leq c\|v\|_{\infty},\\[9pt]
M^{(2)}&=\dps\int_{0}^{y}(x-s)^{-\mu}\frac{|K(y,s)-K(x,s)|}{(x-y)^{\kappa}}|v(s)|ds\\[9pt]
&\dps\leq c\max_{s\in I}\|K(\cdot,s)\|_{0,\kappa}\|v\|_{\infty}\int_{0}^{x}(x-s)^{-\mu}ds\\[9pt]
&\leq c\|v\|_{\infty}B(1-\mu,1)x^{1-\mu}\leq c\|v\|_{\infty},
\eearr
For $M_2,$ we have
\bex
M_{2}\dps\leq c\|v\|_{\infty}(x-y)^{-\kappa}\int_{y}^{x}(x-s)^{-\mu}ds\leq c\|v\|_{\infty}(x-y)^{1-\mu-\kappa}\int_{0}^{1}(1-\tau)^{-\mu}d\tau\leq c\|v\|_{\infty}.
\eex
Gathering all these estimates thus gives \eqref{equ7}.
\end{proof}

We can also prove the following result.
\begin{lemma}\label{lemm2}
For any function $v(x)\in C(I),$ {\color{black}$K\in C(I\times I)$, and $K(\cdot,s)\in C^{0,\kappa}(I)$ with} $0<\kappa<1-\mu$, there exists a positive constant $c$ such that
\bq\label{rev_eq3}
\frac{|\big(\mathcal{K}v\big)(x^{1/\lambda})-\big(\mathcal{K}v\big)(y^{1/\lambda})|}{|x-y|^{\kappa}}\leq c \max_{x\in I}|v(x)|, ~~~\forall x,y\in I,~x\neq y.
\eq
Thus
\begin{equation*}
\|\big(\mathcal{K}v\big)(x^{\frac{1}{\lambda}})\|_{0,\kappa}\leq c\|v\|_{\infty}.
\end{equation*}
\end{lemma}
{\color{black}
\begin{proof}
It follows from \eqref{equ7} that
\bq \label{rev_eq1}
\dps\frac{|\big(\mathcal{K}v\big)(x^{1/\lambda})-\big(\mathcal{K}v\big)(y^{1/\lambda})|}{|x^{1/\lambda}-y^{1/\lambda}|^{\kappa}}\leq c \max_{x\in I}|v(x)|, ~~~\forall x,y\in I,~x\neq y.
\eq
Noticing that, for any $x,y\in I, x\neq y, 0<\lambda\leq1$, we have
\begin{equation*}
|x^{\frac{1}{\lambda}}-y^{\frac{1}{\lambda}}|^{\kappa}\sim O(|x-y|^{\kappa}),\mbox{ or }|x^{\frac{1}{\lambda}}-y^{\frac{1}{\lambda}}|^{\kappa}\sim o(|x-y|^{\kappa}).
\end{equation*}
Thus we obtain
\bq\label{rev_eq2}
\dps\frac{|\big(\mathcal{K}v\big)(x^{1/\lambda})-\big(\mathcal{K}v\big)(y^{1/\lambda})|}{|x-y|^{\kappa}}\leq c\frac{|\big(\mathcal{K}v\big)(x^{1/\lambda})-\big(\mathcal{K}v\big)(y^{1/\lambda})|}{|x^{1/\lambda}-y^{1/\lambda}|^{\kappa}}.
\eq
Combining with \eqref{rev_eq1} and \eqref{rev_eq2} gives \eqref{rev_eq3}. Thus we complete the proof.
\end{proof}}

\section{Fractional Jacobi Spectral-Collocation Method and Convergence Analysis}
\setcounter{equation}{0}

This section is devoted to developing and analyzing an efficient method for
the following equation: $0<\mu<1$,
\begin{equation}\label{eq1}
u(x)=g(x)+ \big(\mathcal{K}u\big)(x),~~x\in I:=[0,1],
\end{equation}
where $\mathcal{K}(\cdot)$ is the weakly singular integral operator defined in \eqref{equ8},
$g(x)\in C(I)$, $K\in C(I\times I), K(x,x)\neq 0$ for $x\in I$.

\subsection{Fractional Jacobi Spectral-Collocation Method for VIEs}

We consider the fractional Jacobi spectral-collocation method as follows:
find fractional polynomial $u_N^\lambda\in P^{\lambda}_{N}(I)$,
such that
\begin{equation*}
u_N^\lambda(x_{i})
=g(x_{i})+\big(\mathcal{K}u_N^\lambda\big)(x_{i}),~~~~0\leq i\leq N,
\end{equation*}
where the collocation points $\{x_{i}\}_{i=0}^{N}$ are roots of $J^{\alpha,\beta,\lambda}_{N+1}(x)$.

Since the exact evaluation of $\big(\mathcal{K}\varphi\big)(x_i)$ is not realizable in pratical cases,
we need to find an efficient way to approximate this term for all
points $\{x_{i}\}_{i=0}^{N}$ without losing the high order accuracy of the scheme.
For a given collocation point $x_i$, we rewrite the integral term by using the variable change
$s=\tau_{i}(\theta):=x_{i}\theta^{\frac{1}{\lambda}}$:
\begin{equation*}
\begin{array}{r@{}l}
\big(\mathcal{K}\varphi\big)(x_{i})&=\dps\int_{0}^{x_{i}}(x_{i}-s)^{-\mu}K(x_{i},s)\varphi(s)ds\\[9pt]
&=\dps\frac{x_{i}^{1-\mu}}{\lambda}\int_{0}^{1}(1-\theta^{\frac{1}{\lambda}})^{-\mu}\theta^{\frac{1}{\lambda}-1}K(x_{i},\tau_{i}(\theta))
\varphi(\tau_{i}(\theta))d\theta.\\[9pt]
&=\dps\frac{x_{i}^{1-\mu}}{\lambda}\int_{0}^{1}(1-\theta)^{-\mu}\theta^{1/\lambda-1}
\Big(\frac{1-\theta^{1/\lambda}}{1-\theta}\Big)^{-\mu}K(x_{i},\tau_{i}(\theta))\varphi(\tau_{i}(\theta))d\theta\\[9pt]
&=\dps \frac{x_{i}^{1-\mu}}{\lambda}\Big(\Big(\frac{1-\theta^{1/\lambda}}{1-\theta}\Big)^{-\mu}K(x_{i},\tau_{i}(\theta)),\varphi(\tau_{i}(\theta))\Big)_{\omega^{-\mu,1/\lambda-1,1}}.
\end{array}
\end{equation*}
Let
\be\label{Kb}
\dps \bar{K}(x_{i},\tau_{i}(\theta))=\frac{x_{i}^{1-\mu}}{\lambda}\Big(\frac{1-\theta^{1/\lambda}}{1-\theta}\Big)^{-\mu}K(x_{i},\tau_{i}(\theta)).
\ee
Then
\bex
\big(\mathcal{K}\varphi\big)(x_{i})
=
(\bar{K}(x_{i},\tau_{i}(\theta)),\varphi(\tau_{i}(\theta)))_{\omega^{-\mu,1/\lambda-1,1}}.
\eex
We choose to approximate the integral $(\bar{K}(x_{i},\tau_{i}(\theta)),\varphi(\tau_{i}(\theta)))_{\omega^{-\mu,1/\lambda-1,1}}$ by
the numerical quandrature $(\bar{K}(x_{i},\tau_{i}(\theta)),\varphi(\tau_{i}(\theta)))_{N,\omega^{-\mu,1/\lambda-1,1}}$, which is defined by
\begin{equation}\label{quad}
(\bar{K}(x_{i},\tau_{i}(\theta)),\varphi(\tau_{i}(\theta)))_{N,\omega^{-\mu,1/\lambda-1,1}}
:=\sum_{j=0}^{N}\bar{K}(x_{i},\tau_{i}(\theta_{j}))\varphi(\tau_{i}(\theta_{j}))\omega_{j},
\end{equation}
where $\{\theta_{j}\}_{j=0}^{N}$ are the zeros of $J^{-\mu,1/\lambda-1,1}_{N+1}(\theta)$ and $\{\omega_{j}\}_{j=0}^{N}$
are the corresponding weights. For the sake of simplification, we will denote
\be\label{MN}
\big(\mathcal{K}_N\varphi\big)(x)
:=
(\bar{K}(x,\tau_{i}(\theta)),\varphi(\tau_{i}(\theta)))_{N,\omega^{-\mu,1/\lambda-1,1}}.
\ee
This leads us to consider the following discrete problem: find
$u_N^{\lambda}(x):=\sum_{i=0}^{N}u_{i}h_{i,\lambda}^{\alpha,\beta}(x)\in P^{\lambda}_{N}(I)$, such that
\begin{equation}\label{equa1}
u_N^\lambda(x_{i})
=g(x_{i})
+
\big(\mathcal{K}_Nu_N^{\lambda}\big)(x_i),~~~~0\leq i\leq N,
\end{equation}
where $\{h_{i,\lambda}^{\alpha,\beta}\}_{i=0}^N$ are the fractional Lagrange polynomials defined in \eqref{eq4}.


\subsection{Convergence analysis}

The purpose of this subsection is to analyze the discrete problem \eqref{equa1} and derive error estimates for the
discrete solution.
First we derive the error estimate in the $L^{\infty}-$norm.
\begin{theorem}\label{the1}
Let $u(x)$ be the exact solution to the Volterra integral equation \eqref{eq1} and
$u_N^{\lambda}(x)$ is the numerical solution of the fractional Jacobi spectral-collocation problem \eqref{equa1}.
Assume $0<\mu<1, -1<\alpha,\beta\leq-\frac{1}{2}$, $K(x,s)\in C^{m}(I,I)$ and $u(x^{\frac{1}{\lambda}})\in B^{m,1}_{\alpha,\beta}(I), m\geq 1$.
Then we have
\begin{equation}\label{ex4}
\|u-u_N^{\lambda}\|_{\infty}\leq
c N^{\frac{1}{2}-m}(\|\partial_{x}^{m}u(x^{\frac{1}{\lambda}})\|_{0,\omega^{\alpha+m,\beta+m,1}}+N^{-\frac{1}{2}}\log N K^{*}\|u\|_{\infty}).
\end{equation}
where
\begin{equation}\label{equati2}
K^{*}=\max_{0\leq i\leq N}\|\partial_{\theta}^{m}\bar{K}(x_{i},\tau_{i}(\cdot))\|_{0,\omega^{m-\mu,m+1/\lambda-1,1}}.
\end{equation}
\end{theorem}
\begin{proof}
Let $e(x)=u(x)-u_N^{\lambda}(x)$ be the error function.
Subtracting \eqref{equa1} from \eqref{eq1} gives the error equation:
\begin{equation}\label{equa5}
\begin{array}{r@{}l}
e_{i}
=\big(\mathcal{K}e\big)(x_{i})+ q_i, \ \ 0\leq i\leq N,
\end{array}
\end{equation}
where $e_i= u(x_i) - u_i$, and $q_i$ is the quadrature error term
\bearr
q_i
& =\big(\mathcal{K}u_N^{\lambda}\big)(x_{i}) - \big(\mathcal{K}_Nu_N^{\lambda}\big)(x_{i}) \\[2mm]
& =(\bar{K}(x_{i},\tau_{i}(\theta)),u_N^{\lambda}(\tau_{i}(\theta)))_{\omega^{-\mu,1/\lambda-1,1}}-(\bar{K}(x_{i},\tau_{i}(\theta)),u_N^{\lambda}(\tau_{i}(\theta)))_{N,\omega^{-\mu,1/\lambda-1,1}},
\eearr
which can be bounded by using Lemma \ref{lem3}
\begin{equation}\label{equa6}
|q_i |\leq cN^{-m}\|\partial_{\theta}^{m}\bar{K}(x_{i},\tau_{i}(\cdot))\|_{0,\omega^{m-\mu,m+1/\lambda-1,1}}\|u_N^{\lambda}(\tau_{i}(\cdot))\|_{0,\omega^{-\mu,1/\lambda-1,1}}.
\end{equation}
Multiplying both sides of \eqref{equa5} by $h^{\alpha,\beta}_{i,\lambda}(x)$ and summing up the resulting equation from $i=0$ to $i=N$
gives
\begin{equation}
I^{\alpha,\beta}_{N,\lambda}u(x)-u_N^{\lambda}(x)=I^{\alpha,\beta}_{N,\lambda}\big(\mathcal{K}e\big)(x)+\sum_{i=0}^{N}q_i h^{\alpha,\beta}_{i,\lambda}(x).
\end{equation}
Rearranging this equation leads to
\begin{equation}\label{equati3}
e(x)=\big(\mathcal{K}e\big)(x)+I_{1}+I_{2}+I_{3},
\end{equation}
where
\begin{equation}\label{equatio1}
I_{1}=u(x)-I^{\alpha,\beta}_{N,\lambda}u(x),~~~I_{2}=\dps\sum_{i=0}^{N}q_i h^{\alpha,\beta}_{i,\lambda}(x), ~~~I_{3}=I^{\alpha,\beta}_{N,\lambda}\big(\mathcal{K}e\big)(x)-\big(\mathcal{K}e\big)(x),
\end{equation}
which correspond respectively the interpolation error, numerical quadrature error, and the interpolation error for the integral operator
$\mathcal{K}$. Using Gronwall lemma gives
\brr\label{eqx}
|e|&\dps\leq |I_{1}+I_{2}+I_{3}|+K_{0}\int_{0}^{x}(x-s)^{-\mu}|e(s)|ds\\[9pt]
&\dps\leq  |I_{1}+I_{2}+I_{3}|+K_{0}\int_{0}^{x}(x-s)^{-\mu}|I_{1}+I_{2}+I_{3}|
\exp\Big(K_{0}\int_{s}^{x}(x-\tau)^{-\mu}d\tau\Big)ds\\[9pt]
&\leq\dps|I_{1}+I_{2}+I_{3}|+K_{0}
\exp\Big(K_{0}\int_{0}^{x}(x-\tau)^{-\mu}d\tau\Big)\int_{0}^{x}(x-s)^{-\mu}|I_{1}+I_{2}+I_{3}|ds\\[9pt]
&\leq\dps|I_{1}+I_{2}+I_{3}|+K_{0}\exp\big(K_{0}/(1-\mu)\big)\int_{0}^{x}(x-s)^{-\mu}|I_{1}+I_{2}+I_{3}|ds.
\err
where $K_{0}=\dps\max_{0\leq s<x\leq1}|K(x,s)|.$  We then obtain
\bq\label{equat3}
\|e\|_{\infty}\leq\dps c(\|I_{1}\|_{\infty}+\|I_{2}\|_{\infty}+\|I_{3}\|_{\infty}),
\eq
{\color{black}where $c=1+\big(K_{0}/(1-\mu)\big) \exp(K_{0}/(1-\mu)\big)$.}
Now we bound the right hand side term by term. For the first term, it follows from
Proposition \ref{th4}:
\be\label{equat4}
\|I_{1}\|_{\infty}
\leq cN^{\frac{1}{2}-m}\|\partial_{x}^{m}u(x^{\frac{1}{\lambda}})\|_{0,\omega^{\alpha+m,\beta+m,\lambda}}.
\ee
For $I_2$, we deduce from \eqref{equa6}
\begin{equation}\label{equat1}
\begin{array}{r@{}l}
\dps\max_{0\leq i\leq N}|q_i |&\dps\leq cN^{-m}\max_{0\leq i\leq N}\|\partial_{\theta}^{m}\bar{K}(x_{i},\tau_{i}(\cdot))\|_{0,\omega^{m-\mu,m+1/\lambda-1,1}}\max_{0\leq i\leq N}\|u_N^{\lambda}(\tau_{i}(\cdot))\|_{0,\omega^{-\mu,1/\lambda-1,1}}\\[9pt]
&\dps\leq cN^{-m}\max_{0\leq i\leq N}\|\partial_{\theta}^{m}\bar{K}(x_{i},\tau_{i}(\cdot))\|_{0,\omega^{m-\mu,m+1/\lambda-1,1}}\|u_N^{\lambda}\|_{\infty}\\[9pt]
&\dps\leq cN^{-m}\max_{0\leq i\leq N}\|\partial_{\theta}^{m}\bar{K}(x_{i},\tau_{i}(\cdot))\|_{0,\omega^{m-\mu,m+1/\lambda-1,1}}(\|e\|_{\infty}+\|u\|_{\infty}).\\
\end{array}
\end{equation}
This, together with Lemma \ref{lemm1}, gives
\begin{equation}\label{equat5}
\begin{array}{r@{}l}
\|I_{2}\|_{\infty}&\leq \|\dps\sum_{i=0}^{N}q_i h^{\alpha,\beta}_{i,1}(x^{\lambda})\|_{\infty}
\dps\leq c\max_{0\leq i\leq N}|q_i |\max_{x\in I}\sum_{i=0}^{N}|h^{\alpha,\beta}_{i,1}(x^{\lambda})|\\[9pt]
&\leq\dps cN^{-m}\log N\max_{0\leq i\leq N}\|\partial_{\theta}^{m}\bar{K}(x_{i},\tau_{i}(\cdot))\|_{0,\omega^{m-\mu,m+1/\lambda-1,1}}(\|e\|_{\infty}+\|u\|_{\infty}).
\end{array}
\end{equation}
 It remains to estimate the third term $I_{3}$. It follows from \eqref{lemm3}, Lemma \ref{lemm2},
 Lemma \ref{lemm1}, and \eqref{equat2} that
\begin{equation}\label{equat6}
\begin{array}{r@{}l}
\|I_{3}\|_{\infty}
&=\dps\max_{x\in I}|I^{\alpha,\beta}_{N,\lambda}\big(\mathcal{K}e\big)(x)-\big(\mathcal{K}e\big)(x)|
=\dps\max_{z^{1/\lambda}=x\in I}|I^{\alpha,\beta}_{N,1}\big(\mathcal{K}e\big)(z^{1/\lambda})-\big(\mathcal{K}e\big)(z^{1/\lambda})|\\[9pt]
&
= \|(I^{\alpha,\beta}_{N,1}-I)\big(\mathcal{K}e\big)(z^{1/\lambda})\|_{\infty}\\[9pt]
&= \|(I^{\alpha,\beta}_{N,1}-I)\big[\big(\mathcal{K}e\big)(z^{1/\lambda})-\mathcal{T}_{N}\big(\mathcal{K}e\big)(z^{1/\lambda})\big]\|_{\infty}\\[9pt]
&\leq (\|I^{\alpha,\beta}_{N,1}\|_{\infty}+1)\|\big(\mathcal{K}e\big)(z^{1/\lambda})-\mathcal{T}_{N}\big(\mathcal{K}e\big)(z^{1/\lambda})\|_{\infty}\\[9pt]
&\leq
cN^{-\kappa}\log N\|e\|_{\infty}, ~~~0< \kappa<1-\mu.

\end{array}
\end{equation}
The desired result then follows from bringing the estimates \eqref{equat4}, \eqref{equat5}, and \eqref{equat6} into
\eqref{equat3}.
This completes the proof.
\end{proof}

\begin{remark}\label{RK1}
Note that the error estimate \eqref{ex4} involves the quantity $K^*$ defined in \eqref{equati2}.
The boundedness of this quantity depends on the regularity of the function $\bar K$, which is a transformation of
$K$ through \eqref{Kb}. Thus it is important to observe that if we set $\lambda=1/p$ with $p$ being a positive integer,
then under the initial regularity assumption on  $K(x,s)$, i.e.,
$K(x,s)\in C^{m}(I,I)$, we have for all $x_i\in I$,
$h(\theta):=K(x_{i},x_{i}\theta^{1/\lambda})=K(x_{i},x_{i}\theta^{p})\in C^{m}(I)$. Furthermore,
noticing that $\Big(\dfrac{1-\theta^{1/\lambda}}{1-\theta}\Big)^{-\mu}\in C^{\infty}(I)$, we then have
$\bar{K}(x_{i},\tau_{i}(\theta))\in C^{m}(I)$. This guarantees
the boundedness of $K^{*}$ in \eqref{equati2}.
{\color{black}
It is also notable that the constant $c$ in front of the estimate \eqref{equat3}
for $\|u-u_N^{\lambda}\|_{\infty}$
will blow up when $\mu\rightarrow1$. We have tried but unfortunately found no way to improve
the estimate given in the proof of Theorem \ref{the1}.
Possible improvement for this constant certainly requires new estimation technique,
and is worth further investigation.}
\end{remark}

To derive the error estimate in the weighted $L^{2}-$norm,
we will make use of the following generalized Hardy's inequality with weights (see \cite{Kuf03,Gog99}).
\begin{lemma}\label{lemm4}
For all measurable function $f\geq 0$, weight functions $u$ and $v$, $1<p\leq q<\infty$, {\color{black}$-\infty\leq a<b\leq\infty$},
the following generalized Hardy's inequality holds
\begin{equation*}
\Big(\int_{a}^{b}|(Tf)(x)|^{q}u(x)dx\Big)^{1/q}
\leq c\Big(\int_{a}^{b}|f(x)|^{p}v(x)dx\Big)^{1/p}
\end{equation*}
if and only if
\begin{equation*}
\sup_{a<x<b}\Big(\int_{x}^{b}u(t)dt\Big)^{1/q}\Big(\int_{a}^{x}v^{1-p'}(t)dt\Big)^{1/p'}<\infty,~~p'=\frac{p}{p-1},
\end{equation*}
where $T$ is an operator of the form
\begin{equation*}
(Tf)(x)=\int_{a}^{x}\rho(x,s)f(s)ds
\end{equation*}
with $\rho(x,s)$ being a given kernel.
\end{lemma}
From Lemma 4.2 in \cite{Che13} and \eqref{equati1}, we have the following weighted mean convergence result
for the fractional interpolation operator $I^{\alpha,\beta}_{N,\lambda}$.
\begin{lemma}\label{lemm5}
For any bounded function $v(x)$ defined on I, there exists a constant $c$ independent of $v$ such that
\begin{equation*}
\dps\sup_{N}\|I^{\alpha,\beta}_{N,\lambda}v\|_{\omega^{\alpha,\beta,\lambda}}
\leq c\|v\|_{\infty}.
\end{equation*}
\end{lemma}

\begin{theorem}\label{the2}
Let $u(x)$ be the exact solution to \eqref{eq1}, $u_N^{\lambda}(x)$ is the solution of discrete problem \eqref{equa1}.
If $0<\mu<1, -1<\alpha, \beta\leq-\frac{1}{2}$, $K(x,s)\in C^{m}(I,I)$, and
$u(x^{\frac{1}{\lambda}})\in B^{m,1}_{\alpha,\beta}(I), m\geq 1$, then we have
\begin{equation}\label{equati4}
\|u-u_N^{\lambda}\|_{0,\omega^{\alpha,\beta,\lambda}}
\leq
c N^{-m}\big[(1+N^{\frac{1}{2}-\kappa})\|\partial_{x}^{m}u(x^{\frac{1}{\lambda}})\|_{0,\omega^{\alpha+m,\beta+m,1}}
+ K^{*}\|u\|_{\infty}\big],\ \ 0<\kappa<1-\mu,
\end{equation}
where $K^{*}$ is defined in \eqref{equati2}.
\end{theorem}
\begin{proof}
Similar to \eqref{eqx}, we deduce from Lemma \ref{lemm4}
\begin{equation}
\begin{array}{r@{}l}
\|e\|_{0,\omega^{\alpha,\beta,\lambda}}
&\dps\leq \|I_{1}\|_{0,\omega^{\alpha,\beta,\lambda}}+\|I_{2}\|_{0,\omega^{\alpha,\beta,\lambda}}+\|I_{3}\|_{0,\omega^{\alpha,\beta,\lambda}} \\[9pt] &~~~\dps+K_{0}\exp(K_{0}/(1-\mu))\Big\|\int_{0}^{x}(x-s)^{-\mu}|I_{1}+I_{2}+I_{3}|ds\Big\|_{0,\omega^{\alpha,\beta,\lambda}}\\[9pt]
&\leq c(\|I_{1}\|_{0,\omega^{\alpha,\beta,\lambda}}+\|I_{2}\|_{0,\omega^{\alpha,\beta,\lambda}}+\|I_{3}\|_{0,\omega^{\alpha,\beta,\lambda}}).\\
\end{array}
\end{equation}
To bound the right hand side, we first use Proposition \ref{th3} to get
\begin{equation}\label{equati5}
\|I_{1}\|_{0,\omega^{\alpha,\beta,\lambda}}=\|u-I^{\alpha,\beta}_{N,\lambda}u\|_{0,\omega^{\alpha,\beta,\lambda}}\leq cN^{-m}\|\partial_{x}^{m}u(x^{\frac{1}{\lambda}})\|_{0,\omega^{\alpha+m,\beta+m,1}}.
\end{equation}
Then applying Lemma \ref{lemm5} and \eqref{equat1} gives
\begin{equation}\label{equati6}
\begin{array}{r@{}l}
\|I_{2}\|_{0,\omega^{\alpha,\beta,\lambda}}
&\dps=\Big\|\sum_{i=0}^{N}q_i h_{i,1}^{\alpha,\beta}(x^{\lambda})\Big\|_{0,\omega^{\alpha,\beta,\lambda}}
\leq c\max_{0\leq i\leq N}|q_i |
\leq cN^{-m}K^{*}(\|e\|_{\infty}+\|u\|_{\infty}).\\
\end{array}
\end{equation}
Finally, we notice $\|\cdot\|_{\omega^{\alpha,\beta,1}}$ can be bounded by $\|\cdot\|_{\infty}$, and
\begin{equation*}
\begin{array}{r@{}l}
\|(I^{\alpha,\beta}_{N,\lambda}-I)\big(\mathcal{K}e\big)(x)\|_{\omega^{\alpha,\beta,\lambda}}
&\dps=\Big\{\int_{0}^{1}\big[(I^{\alpha,\beta}_{N,\lambda}-I)\big(\mathcal{K}e\big)(x)\big]^{2}
\lambda(1-x^{\lambda})^{\alpha}x^{(\beta+1)\lambda-1}dx\Big\}^{1/2}\\[9pt]
&\dps=\Big\{\int_{0}^{1}\Big[\sum_{i=0}^{N}\big(\mathcal{K}e\big)(x_{i})h_{i,\lambda}^{\alpha,\beta}(x)-\big(\mathcal{K}e\big)(x)\Big]^{2}
(1-x^{\lambda})^{\alpha}x^{\beta\lambda}dx^{\lambda}\Big\}^{1/2}\\[9pt]
&\dps=\Big\{\int_{0}^{1}\Big[\sum_{i=0}^{N}\big(\mathcal{K}e\big)(z_{i}^{\frac{1}{\lambda}})h_{i,1}^{\alpha,\beta}(z)
-\big(\mathcal{K}e\big)(z^{\frac{1}{\lambda}})\Big]^{2}(1-z)^{\alpha}z^{\beta}dz\Big\}^{1/2}\\[9pt]
&\dps=\|(I^{\alpha,\beta}_{N,1}-I)\big(\mathcal{K}e\big)(x^{\frac{1}{\lambda}})\|_{\omega^{\alpha,\beta,1}},
\end{array}
\end{equation*}
where $\{z_{i}\}_{i=0}^{N}$ are zeros of shifted Jacobi polynomial {\color{black}$J^{\alpha,\beta,1}_{N+1}(z).$}
Thus it follows from \eqref{equat2}, Lemma \ref{lemm2}, and Lemma \ref{lemm5}
\begin{equation}\label{pequat1}
\begin{array}{r@{}l}
\|I_{3}\|_{\omega^{\alpha,\beta,\lambda}}
&=\|(I^{\alpha,\beta}_{N,\lambda}-I)\big(\mathcal{K}e\big)(x)\|_{\omega^{\alpha,\beta,\lambda}}
=\|(I^{\alpha,\beta}_{N,1}-I)\big(\mathcal{K}e\big)(x^{\frac{1}{\lambda}})\|_{\omega^{\alpha,\beta,1}}\\[9pt]
&=\big\|(I^{\alpha,\beta}_{N,1}-I)\big[\big(\mathcal{K}e\big)(x^{\frac{1}{\lambda}})-\mathcal{T}_{N}\big(\mathcal{K}e\big)(x^{\frac{1}{\lambda}})\big]\big\|_{\omega^{\alpha,\beta,1}}\\[9pt]
&\leq \big\|I^{\alpha,\beta}_{N,1}\big[\big(\mathcal{K}e\big)(x^{\frac{1}{\lambda}})-\mathcal{T}_{N}\big(\mathcal{K}e\big)(x^{\frac{1}{\lambda}})\big]\big\|_{\omega^{\alpha,\beta,1}}
+\big\|\big(\mathcal{K}e\big)(x^{\frac{1}{\lambda}})-\mathcal{T}_{N}\big(\mathcal{K}e\big)(x^{\frac{1}{\lambda}})\big\|_{\omega^{\alpha,\beta,1}}\\[9pt]
&\leq c\big\|\big(\mathcal{K}e\big)(x^{\frac{1}{\lambda}})-\mathcal{T}_{N}\big(\mathcal{K}e\big)(x^{\frac{1}{\lambda}})\big\|_{\infty}\\[9pt]
&\leq cN^{-\kappa}\|\big(\mathcal{K}e\big)(x^{\frac{1}{\lambda}})\|_{0,\kappa}\\[9pt]
&\leq cN^{-\kappa}\|e\|_{\infty},0<\kappa<1-\mu.
\end{array}
\end{equation}
Then we deduce from Theorem \ref{the1}
\begin{equation}\label{equati7}
\|I_{3}\|_{\omega^{\alpha,\beta,\lambda}}\leq
C N^{\frac{1}{2}-m-\kappa}(\|\partial_{x}^{m}u(x^{\frac{1}{\lambda}})\|_{0,\omega^{\alpha+m,\beta+m,1}}+N^{-\frac{1}{2}}\log N K^{*}\|u\|_{\infty}).
\end{equation}
Combining \eqref{equati5}, \eqref{equati6}, and \eqref{equati7} proves the theorem.
\end{proof}

\section{Numerical results}
\setcounter{equation}{0}

In this section, we present some numerical results to verify the error estimates obtained in
the previous section for the proposed method.
The main goal is to demonstrate that the new method has the capability to capture the typical solutions of VIEs
with high accuracy. To this end a series of
numerical tests are carried out in the following.
All the errors to be presented below are in $L^{\infty}(I)-$ and $L^{2}_{\omega^{\alpha,\beta,\lambda}}(I)-$norms,
where $\alpha,~\beta$ are related to the corresponding collocation points,
i.e., the zeros of $J^{\alpha,\beta,\lambda}_{N+1}(x)$.
In order to calculate the errors, for the test cases in which the exact solution is not available,
a numerical solution obtained with a very large $N$ will be served as the exact solution.

\begin{example}\label{ex1}
 Consider the Volterra integral equation \eqref{eq1}
with $g(x)=1$ and $K(x,s)=\exp(x-s).$
\end{example}
It has been known from \cite{Bru04} that for this smooth $g$ and kernel function $K$,
the solution of \eqref{eq1} can be expressed as
\begin{equation}\label{ex5}
u(x)=\dps\sum_{(j,k)\in G}\gamma_{j,k}x^{j+k\mu}+u_{r}(x),~~~~x\in I,
\end{equation}
where $G:=\{(j,k): j, k$ are non-negative integers\}, $\gamma_{j,k}$ are constants, and $u_{r}$ is a smooth function.

In virtue of the solution structure, we want to choose the value of $\lambda$ such that
$u(x^{1/\lambda})$ is smooth or as regular as possible. This can be easily done by taking $\lambda$ such that
$(j+k\mu)/\lambda$ is integer or as large as possible.
In this first example, we test for $\mu=0.2$ and $0.5$. In these cases there are many possible choices for
$\lambda$. The obtained numerical results
are shown in semi-log scale as a function of $N$ in Figure \ref{fig1} with $\alpha=\beta=-\frac{1}{2}$.
It is clear that the exponential convergence rate is attained with $\lambda=1/5, 1/10, 1/15$ for $\mu=0.2$, and
$\lambda=1/2, 1/4, 1/6$ for $\mu=0.5$ since all these values make $(j+k\mu)/\lambda$ integer.

\begin{figure*}[htbp]
\begin{minipage}[t]{0.32\linewidth}
\centerline{\includegraphics[scale=0.33]{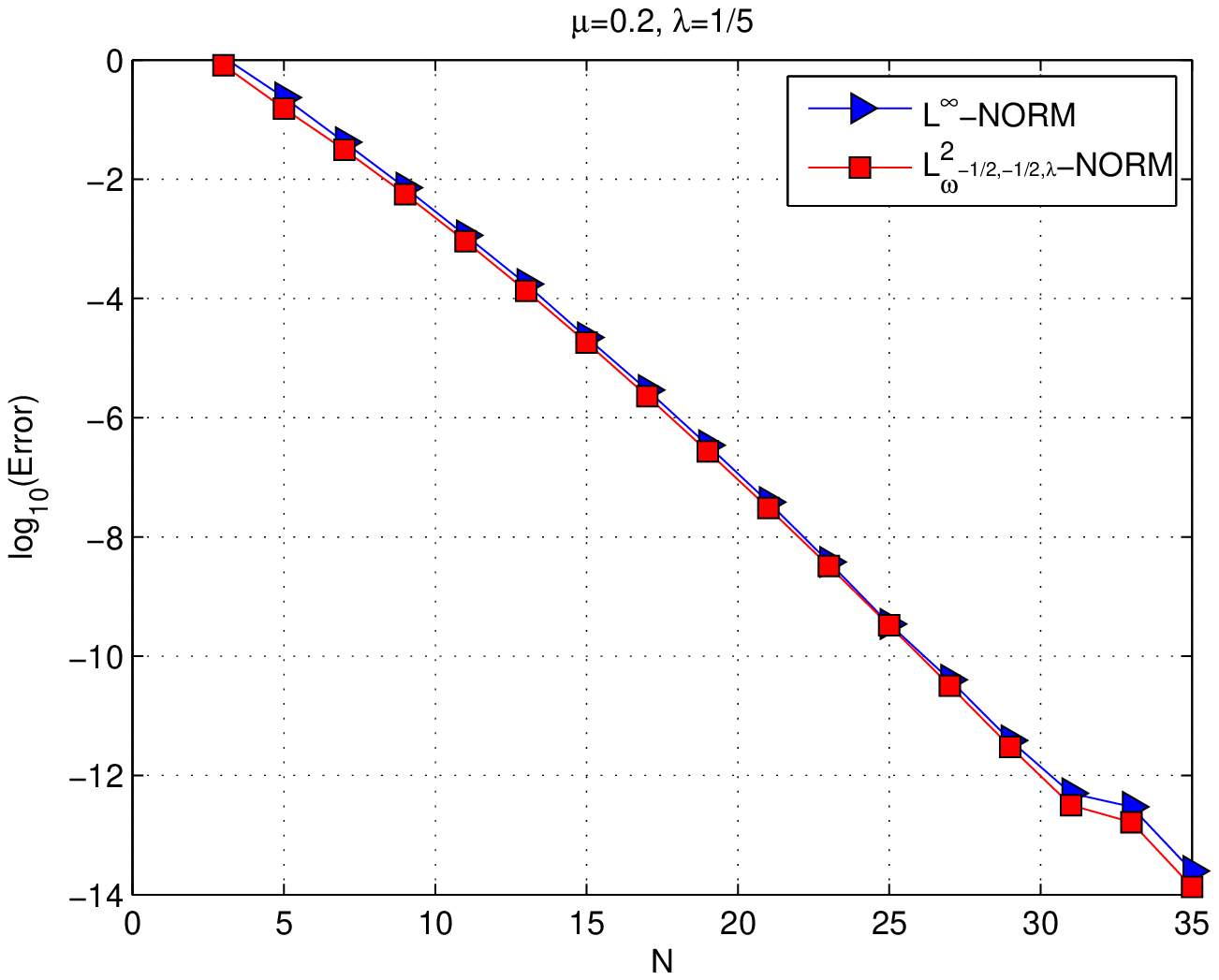}}
\centerline{(a)}
\end{minipage}
\begin{minipage}[t]{0.32\linewidth}
\centerline{\includegraphics[scale=0.33]{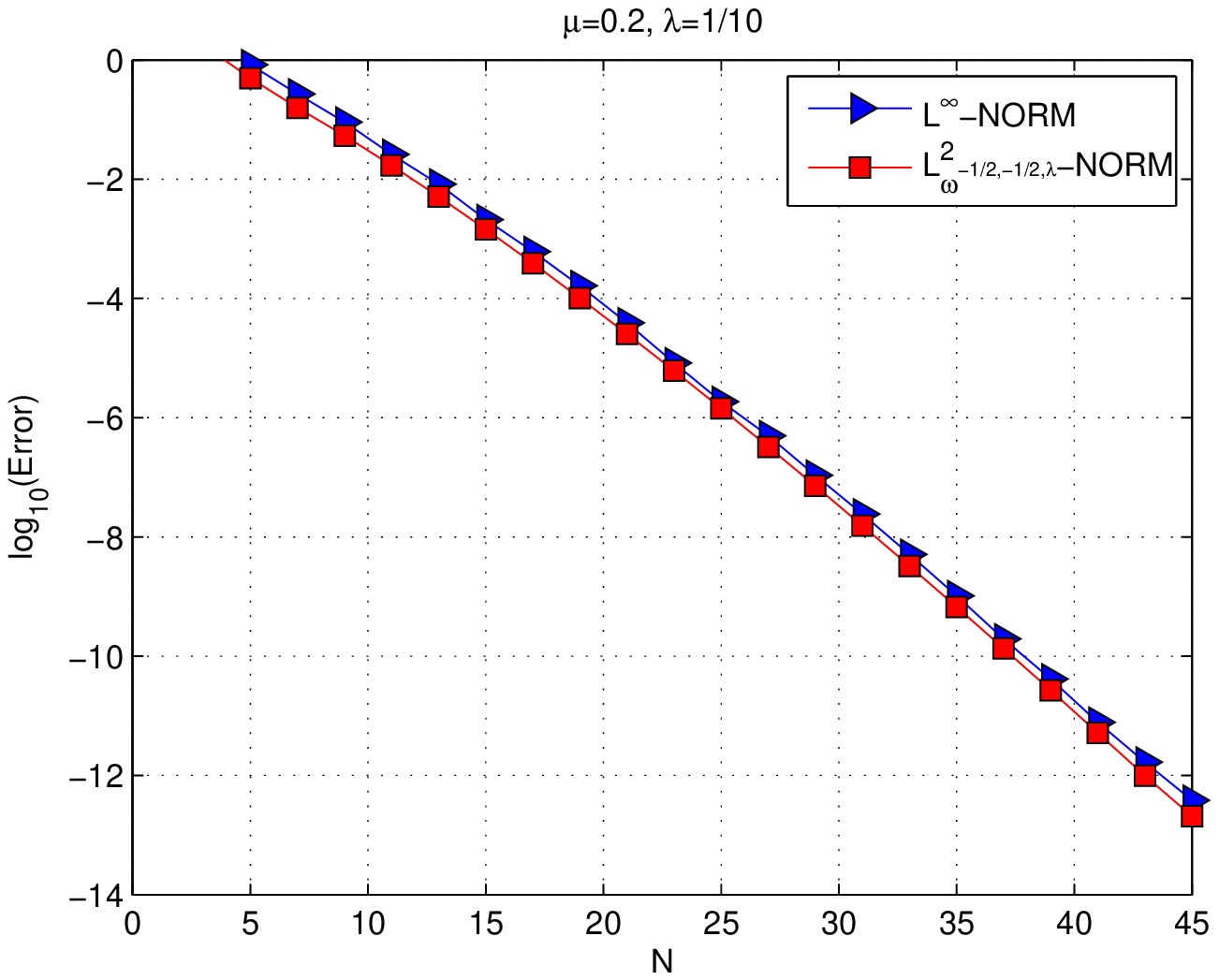}}
\centerline{(b)}
\end{minipage}
\begin{minipage}[t]{0.32\linewidth}
\centerline{\includegraphics[scale=0.33]{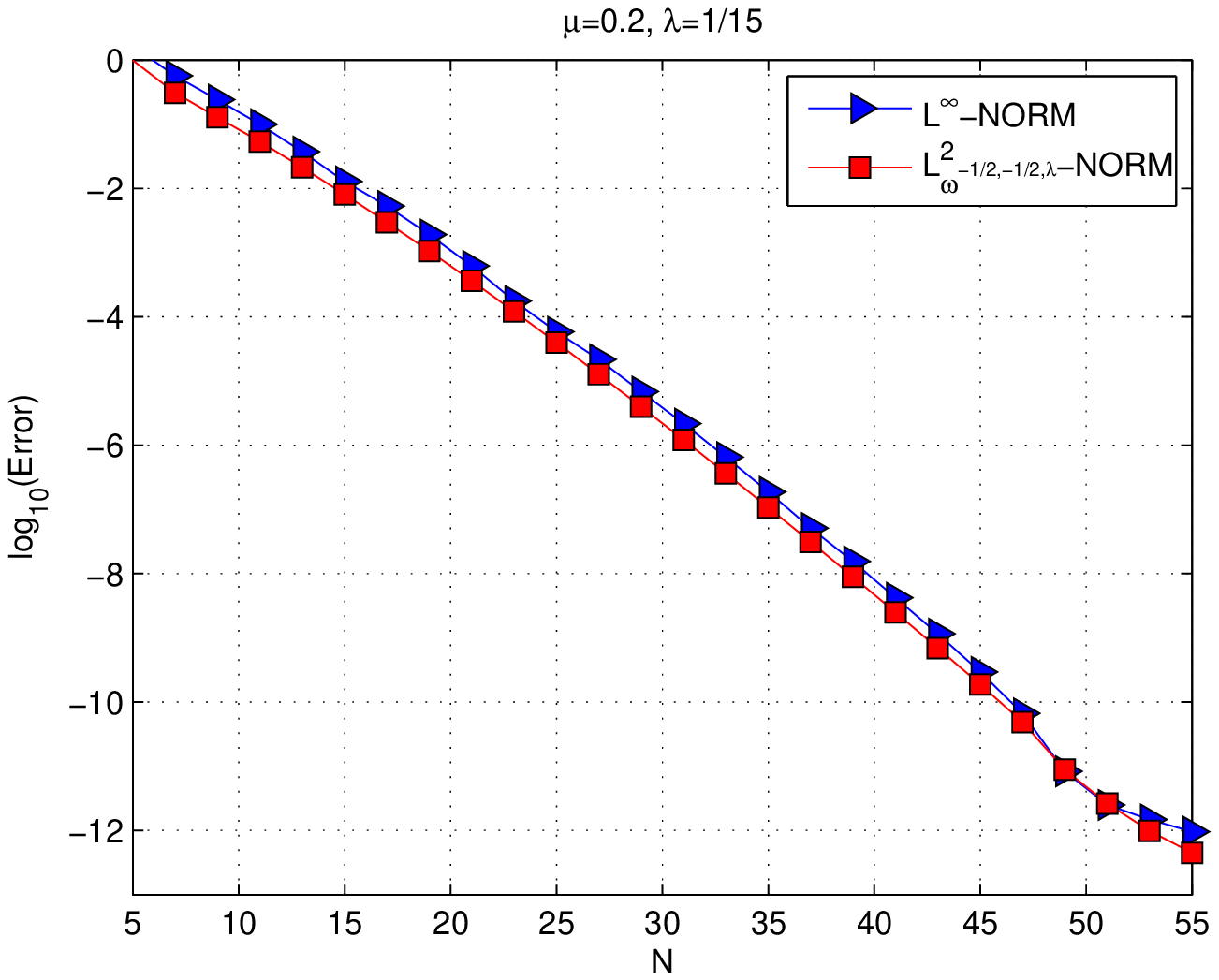}}
\centerline{(c)}
\end{minipage}
\vfill
\vskip 3mm
\begin{minipage}[t]{0.32\linewidth}
\centerline{\includegraphics[scale=0.33]{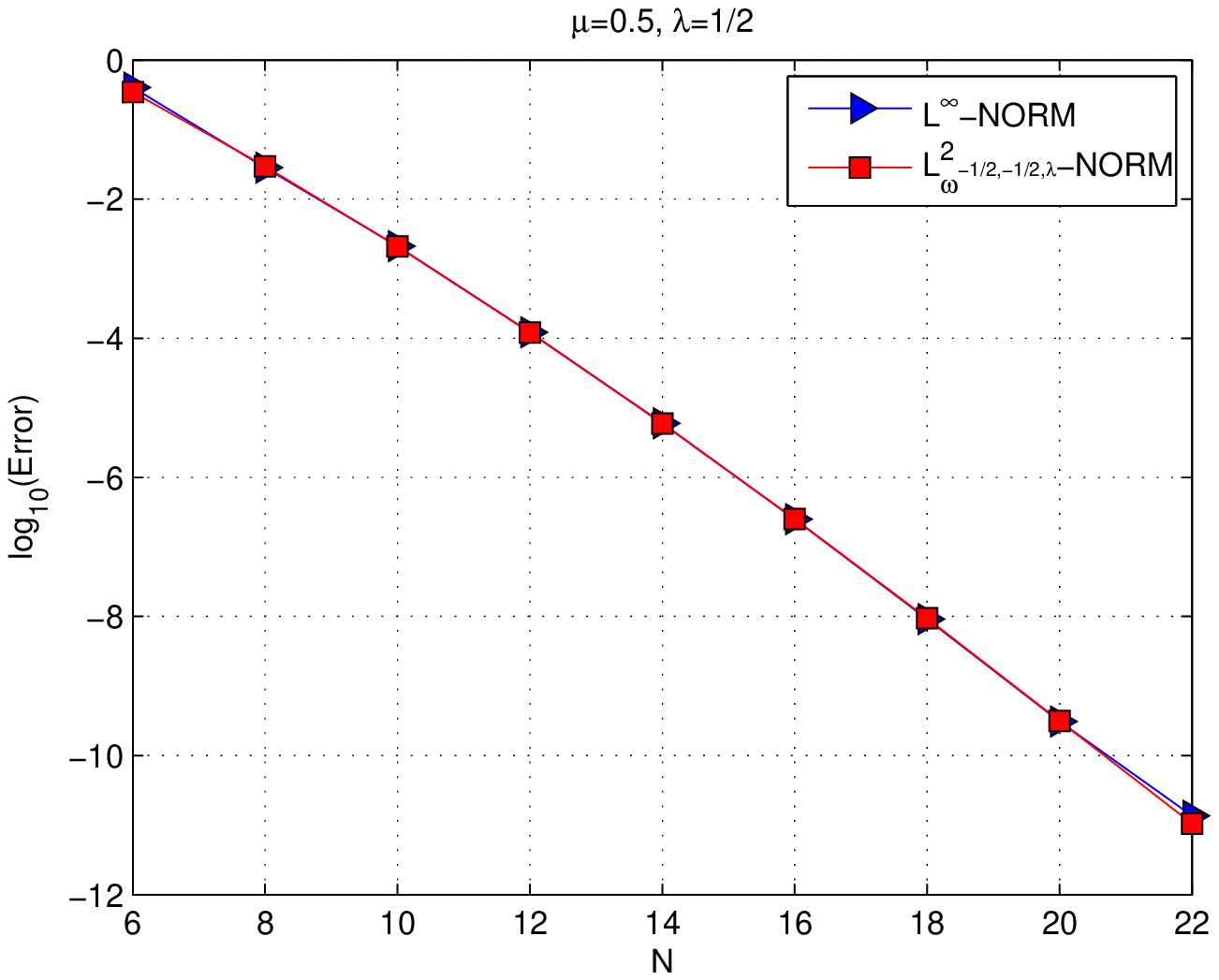}}
\centerline{(d)}
\end{minipage}
\begin{minipage}[t]{0.32\linewidth}
\centerline{\includegraphics[scale=0.33]{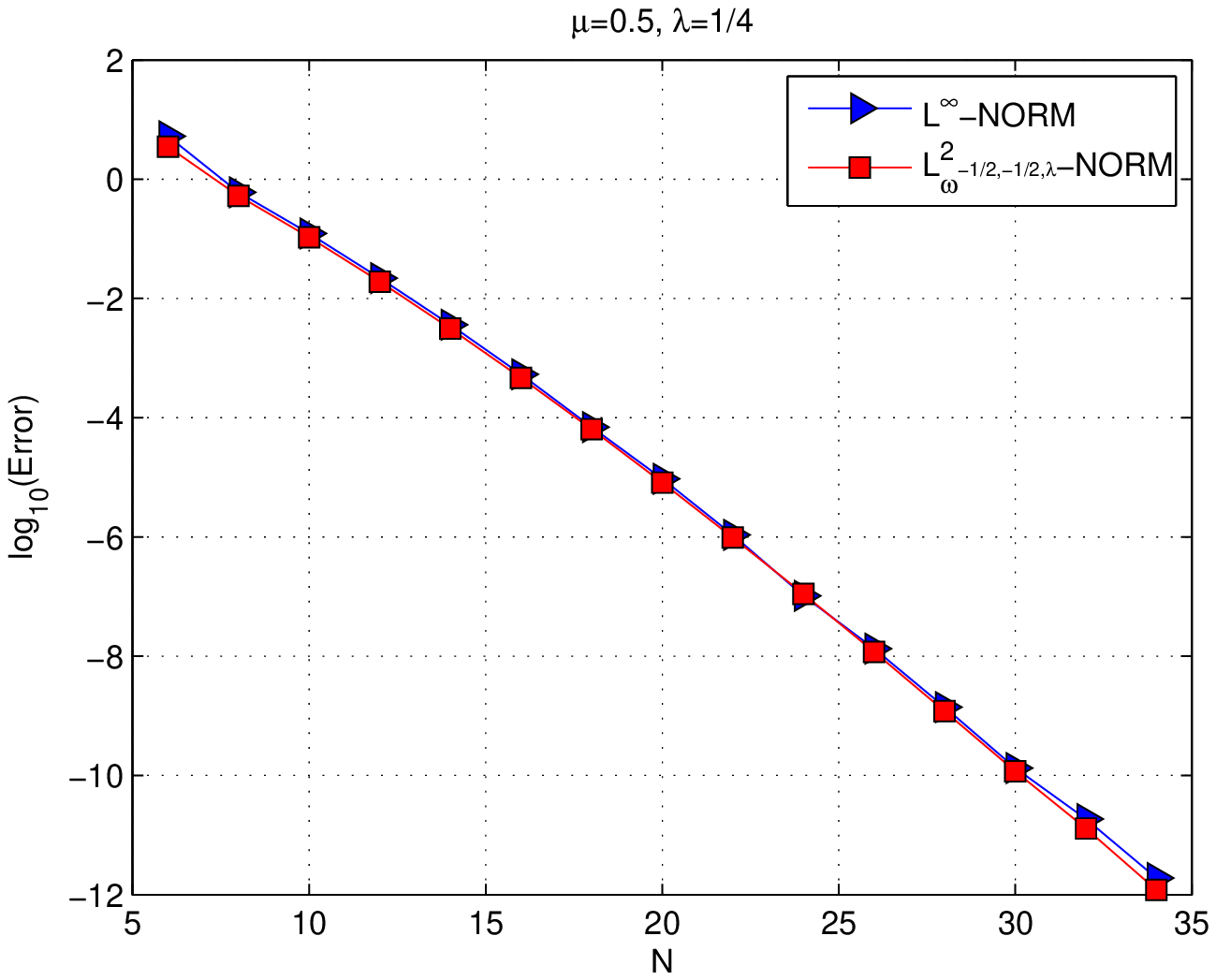}}
\centerline{(e)}
\end{minipage}
\begin{minipage}[t]{0.32\linewidth}
\centerline{\includegraphics[scale=0.33]{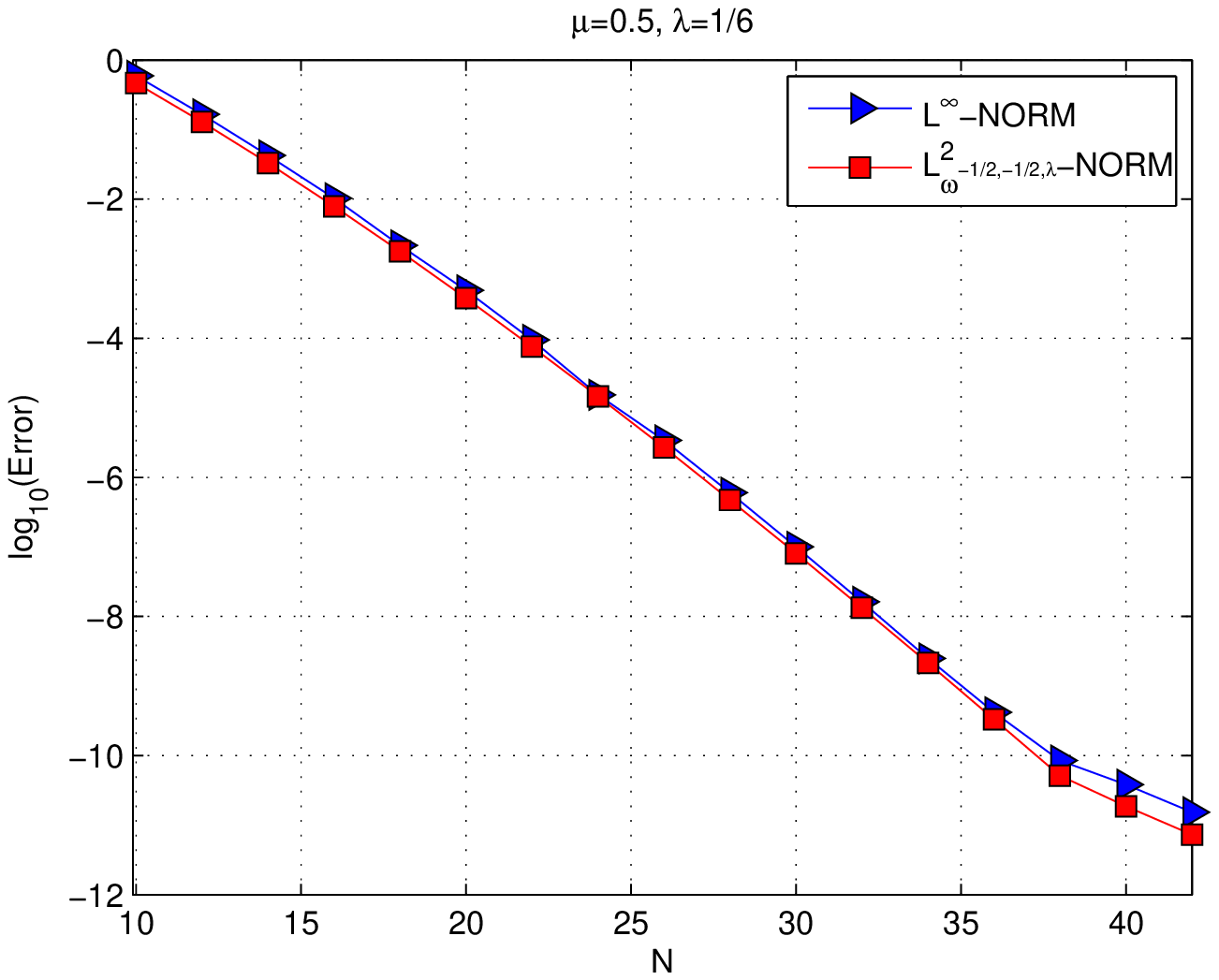}}
\centerline{(f)}
\end{minipage}
\caption{(Example \ref{ex1}) $L^{\infty}-$ and $L^{2}_{\omega^{-1/2,-1/2,\lambda}}-$norm errors versus $\lambda-$polynomial degree $N$ with:
(a) $\mu=0.2,\lambda=\frac{1}{5}$;
(b) $\mu=0.2,\lambda=\frac{1}{10}$;
(c) $\mu=0.2,\lambda=\frac{1}{15}$;
(d) $\mu=0.5,\lambda=\frac{1}{2}$;
(e) $\mu=0.5,\lambda=\frac{1}{4}$;
(f) $\mu=0.5,\lambda=\frac{1}{6}$.
}\label{fig1}
\end{figure*}

\begin{example}\label{ex2}
 Consider the equation \eqref{eq1} with $K(x,s)=1,$ and limited regular source term $g(x)=x^{0.5}$.
\end{example}
In this case, there is not available result on the regularity of the exact solution. So it is interesting to see if high accuracy
is attainable for the fractional spectral method by choosing suitable $\lambda$.
Figure \ref{fig2} shows the errors in semi-log scale as a function of
$N$ for a number of values of $\mu$ and $\lambda$.
In computing the errors we have assumed that the numerical solution obtained with $N=80$
is the ``exact" solution since the real solution is unknown.
Surprisingly, for all tested values of $\mu$ and $\lambda$,
the error curves are almost linear versus the degrees of the $\lambda$-polynomial. This means that
the convergence rate of the proposed method is exponential even if there is no information about the exact solution.
A reasonable explanation for this excellent result
is that the transformed solution, i.e., $u(x^{1/\lambda})$, becomes smooth or regular enough if a suitably small $\lambda$
is used in the approximation.

\begin{figure*}[htbp]
\begin{minipage}[t]{0.45\linewidth}
\centerline{\includegraphics[scale=0.45]{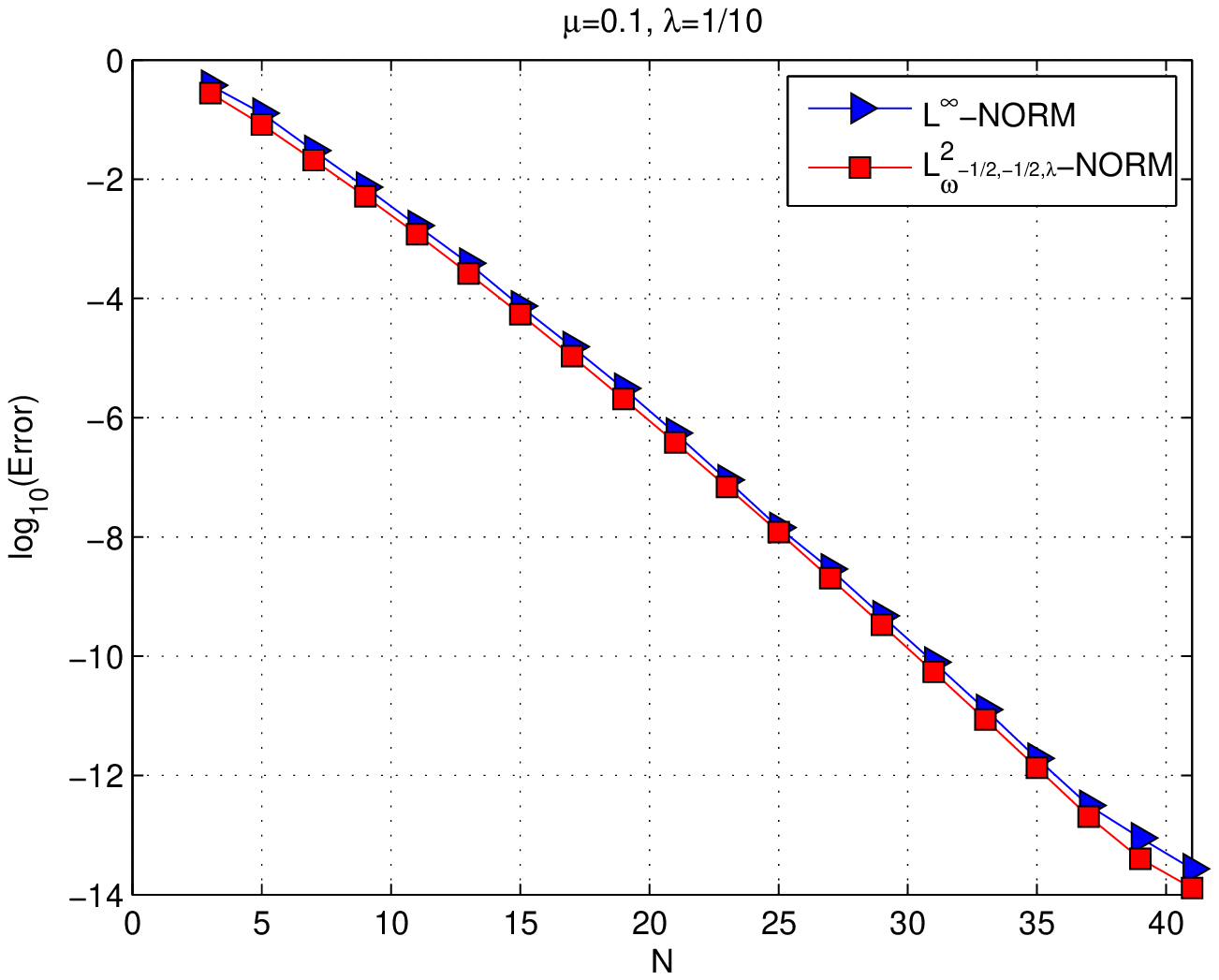}}
\centerline{(a)}
\end{minipage}
\begin{minipage}[t]{0.45\linewidth}
\centerline{\includegraphics[scale=0.45]{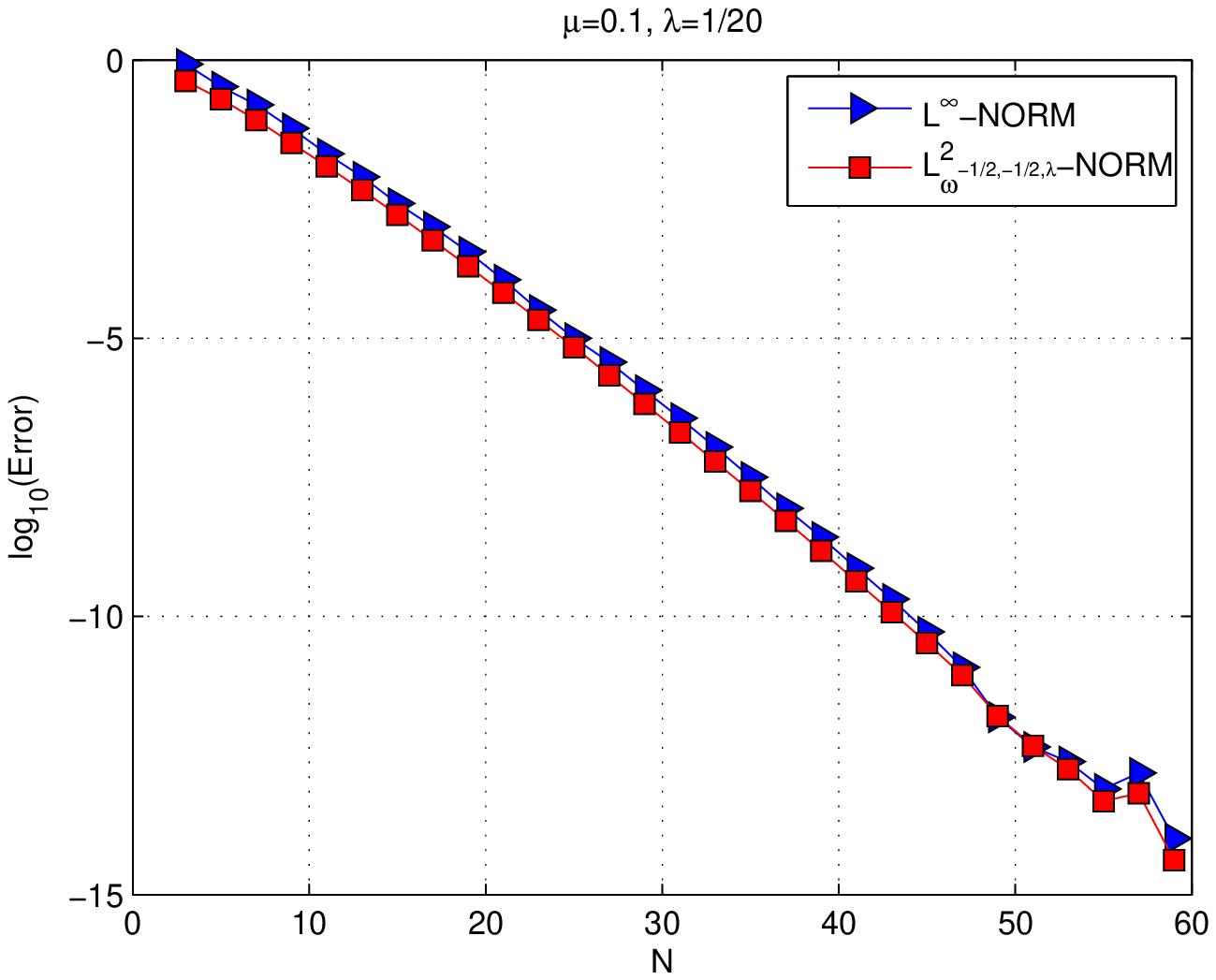}}
\centerline{(b)}
\end{minipage}
\vfill
\vskip 3mm
\begin{minipage}[t]{0.45\linewidth}
\centerline{\includegraphics[scale=0.45]{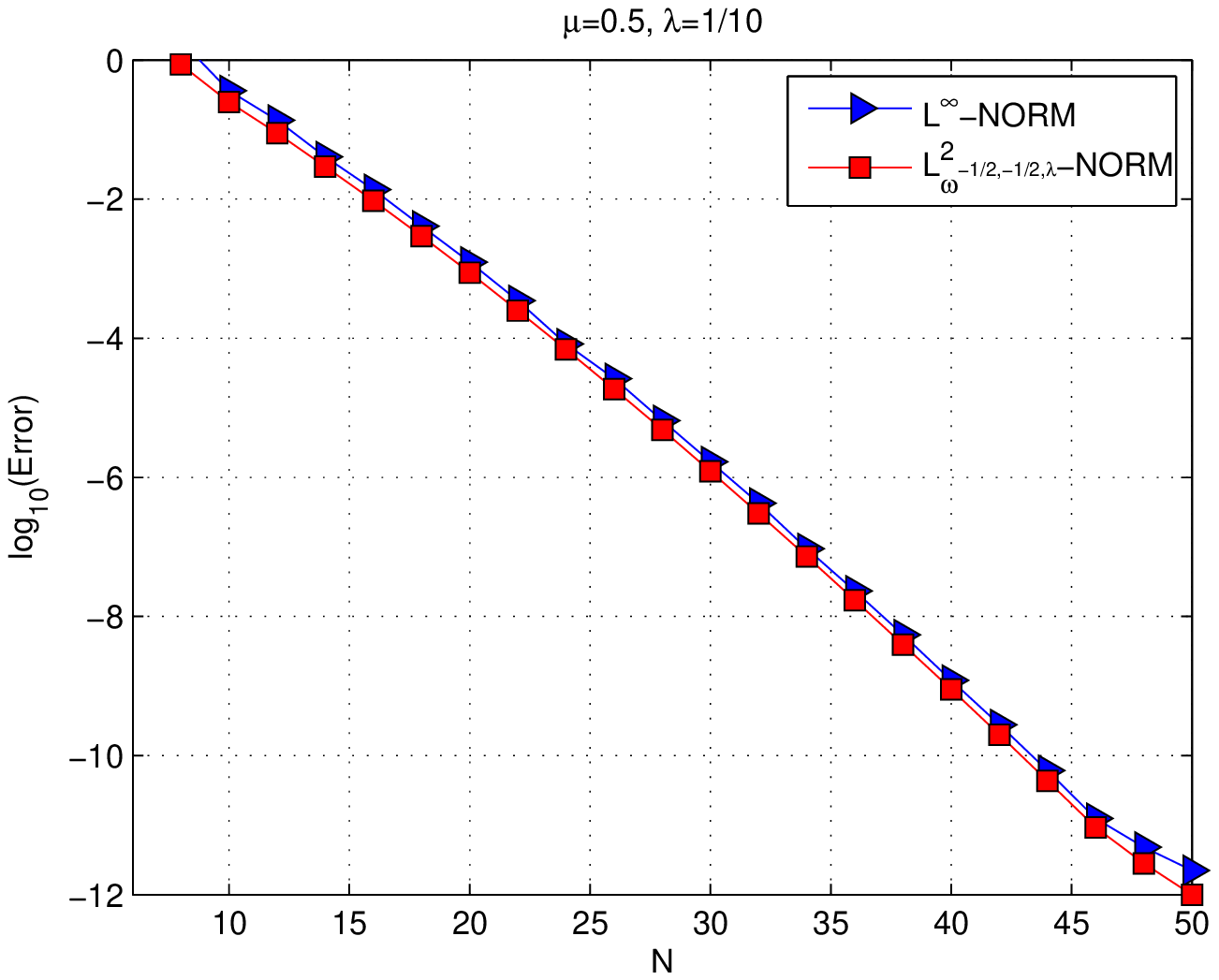}}
\centerline{(c)}
\end{minipage}
\begin{minipage}[t]{0.45\linewidth}
\centerline{\includegraphics[scale=0.45]{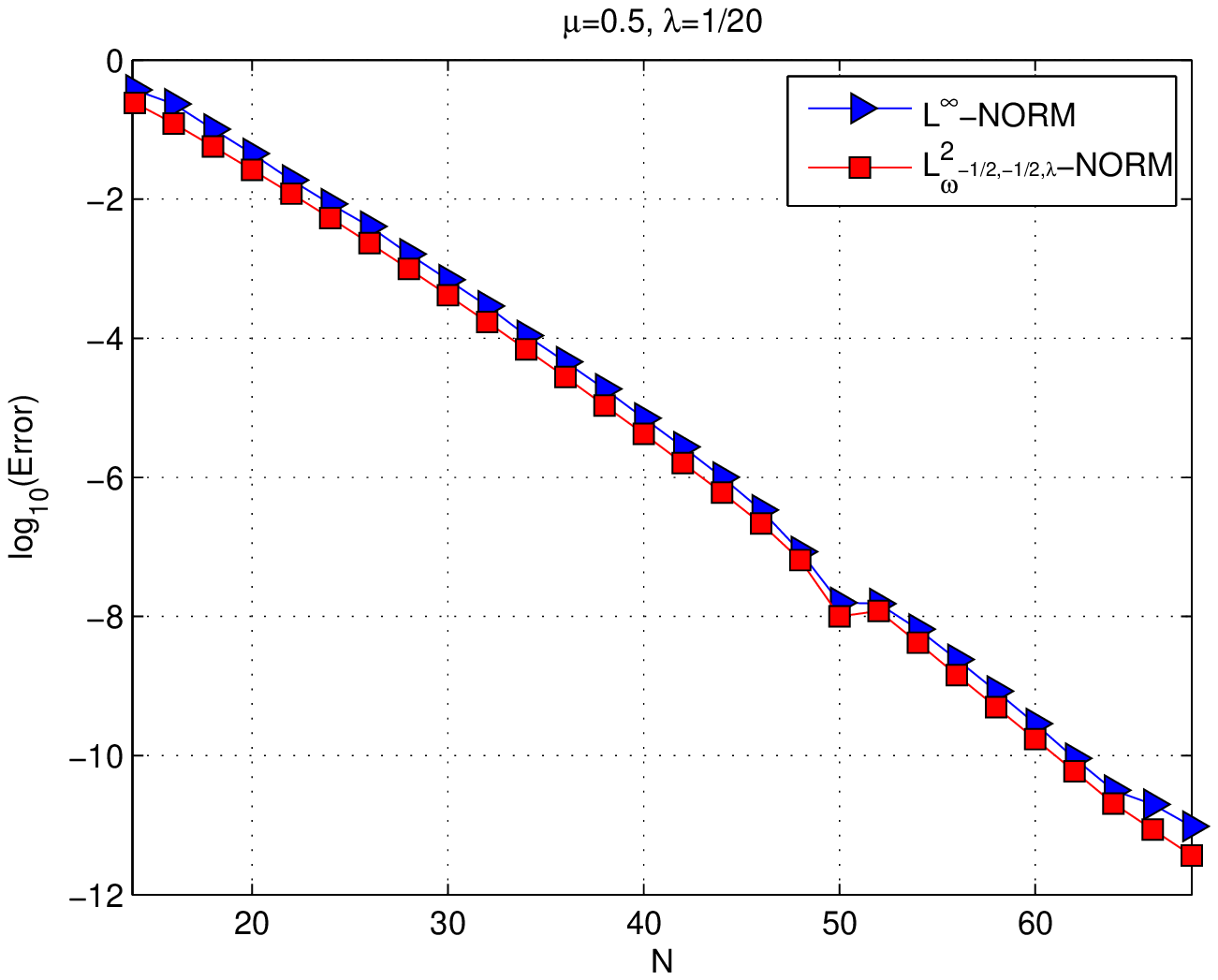}}
\centerline{(d)}
\end{minipage}
\caption{(Example \ref{ex2}) Errors versus $\lambda-$polynomial degree $N$ with
(a) $\mu=0.1,\lambda=\frac{1}{10}$;
(b) $\mu=0.1,\lambda=\frac{1}{20}$;
(c) $\mu=0.5,\lambda=\frac{1}{10}$;
(d) $\mu=0.5,\lambda=\frac{1}{20}$.
}\label{fig2}
\end{figure*}

\begin{example} (\cite{Li15} Example 5.4)\label{pex3}
Consider the fabricated solution $u(x)=x^{-\mu} \sin(x)$, corresponding to $K(x,s)=1$ and
\begin{equation*}
g(x)=x^{-\mu} \sin(x)+\sqrt{\pi}\Gamma(1-\mu)x^{\frac{1}{2}-\mu}\sin(\frac{x}{2}){\color{black}\mathcal{B}(\frac{1}{2}-\mu,\frac{x}{2})},
\end{equation*}
 where {\color{black}$\mathcal{B}(\cdot,\cdot)$} is the Bessel function, i.e.,
 \begin{equation*}
 {\color{black}\mathcal{B}(\mu,x)}=(\frac{x}{2})^{\mu}\sum_{k=0}^{+\infty}\frac{(-x^{2})^{k}}{k!\Gamma(\mu+k+1)4^{k}}.
 \end{equation*}
\end{example}

Clearly, the exact solution of this example has singularity at the left end point,
i.e., $u'(x) \sim x^{-\mu}$ at $x = 0$.
In Figure \ref{fig3}, we plot the errors in the $L^{\infty}(I)$- and $L^{2}_{\omega^{-1/2,-1/2,\lambda}}$-norm
in semi-log scale as a function of $N$ for $\mu=\frac{1}{2},\frac{2}{3}$. In the calculation we have used
$\lambda=\frac{1}{2}$ and $\frac{1}{4}$ for $\mu=\frac{1}{2}$, and
$\lambda=\frac{1}{3}$ and $\frac{1}{6}$ for $\mu=\frac{2}{3}$.
We observe that the exponential convergence is obtained by the proposed fractional spectral method, which
is an expected result since for the tested $\mu$, $u(x^{1/\lambda})$ is smooth for the $\lambda$ values used
in the computation.
It is worth to notice that this example has also been investigated in \cite{Li15} by using a classical spectral method.
The strategy used in that paper is a smoothing transformation approach.
That is, the original integral equation is first transformed into a new one
having smooth solution, then a pseudo spectral method based on the classical
Chebyshev or Legendre polynomials is constructed to approximate the smooth solution.
We emphasize that the current paper makes use of a new idea, which is completely different than
the smoothing transformation approach used in \cite{Li15}. The new idea allows full analysis
for the proposed numerical schemes, and is applicable to a wider class of problems.

\begin{figure*}[htbp]
\begin{minipage}[t]{0.45\linewidth}
\centerline{\includegraphics[scale=0.45]{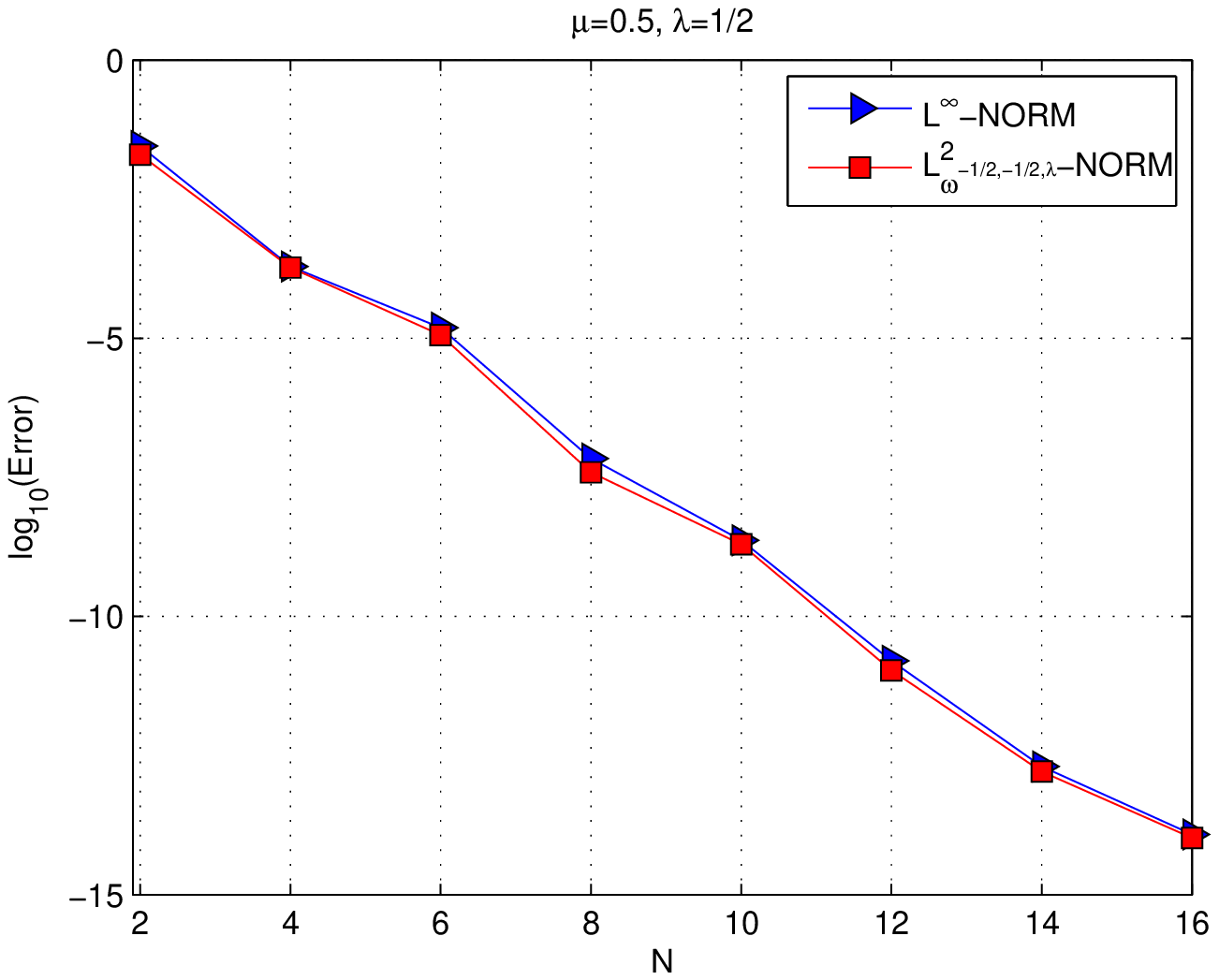}}
\centerline{(a)}
\end{minipage}
\begin{minipage}[t]{0.45\linewidth}
\centerline{\includegraphics[scale=0.45]{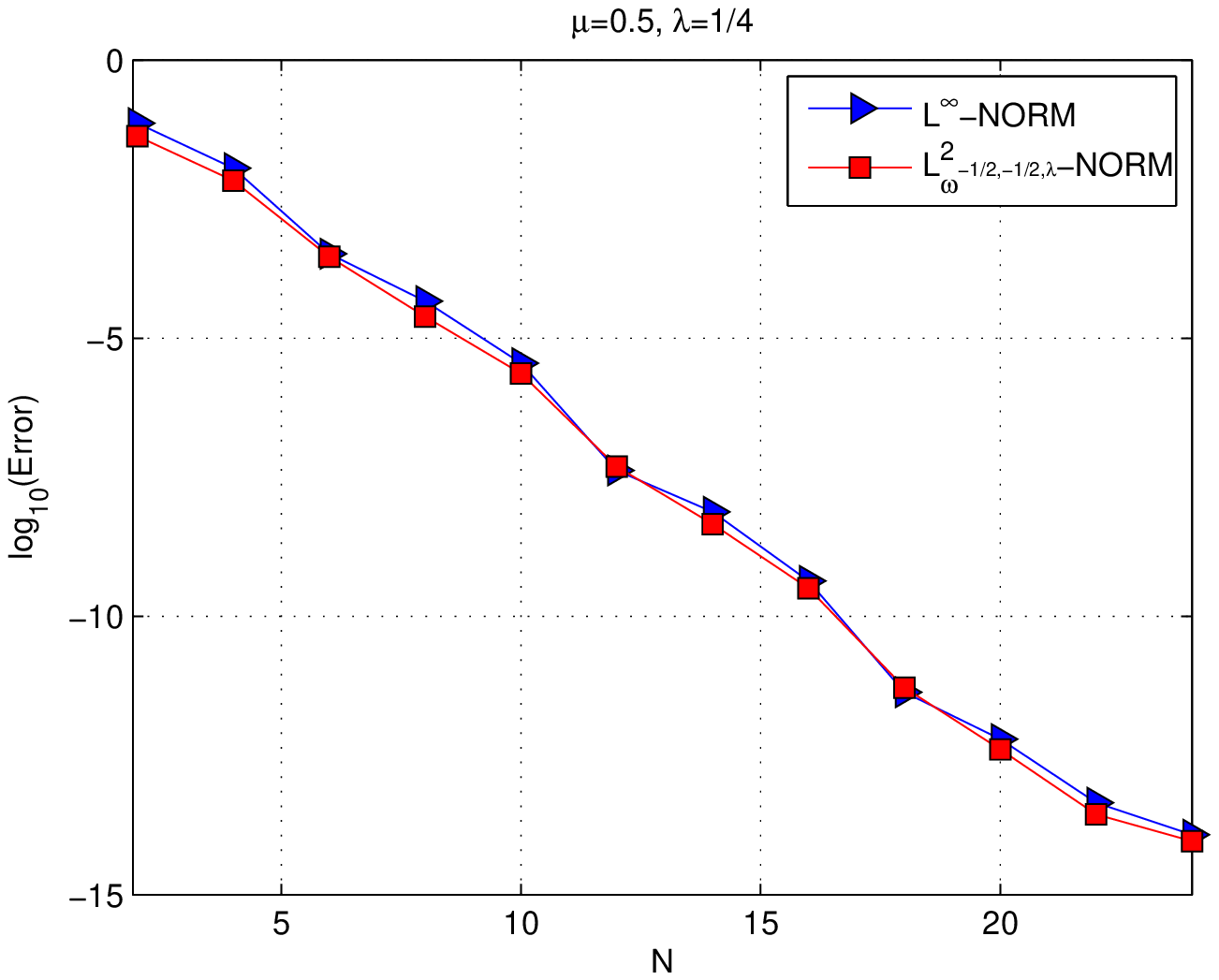}}
\centerline{(b)}
\end{minipage}
\vfill
\vskip 3mm
\begin{minipage}[t]{0.45\linewidth}
\centerline{\includegraphics[scale=0.45]{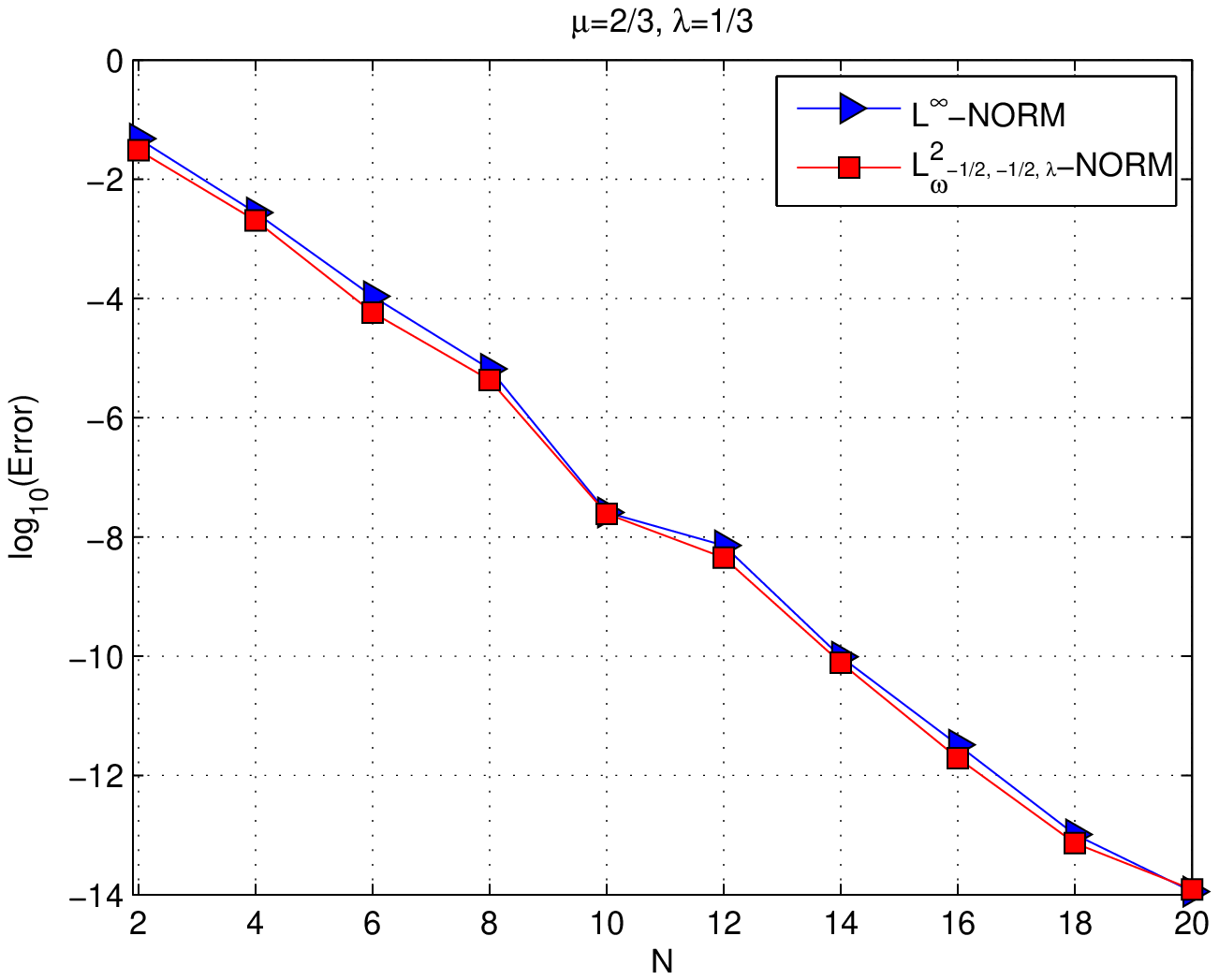}}
\centerline{(c)}
\end{minipage}
\begin{minipage}[t]{0.45\linewidth}
\centerline{\includegraphics[scale=0.45]{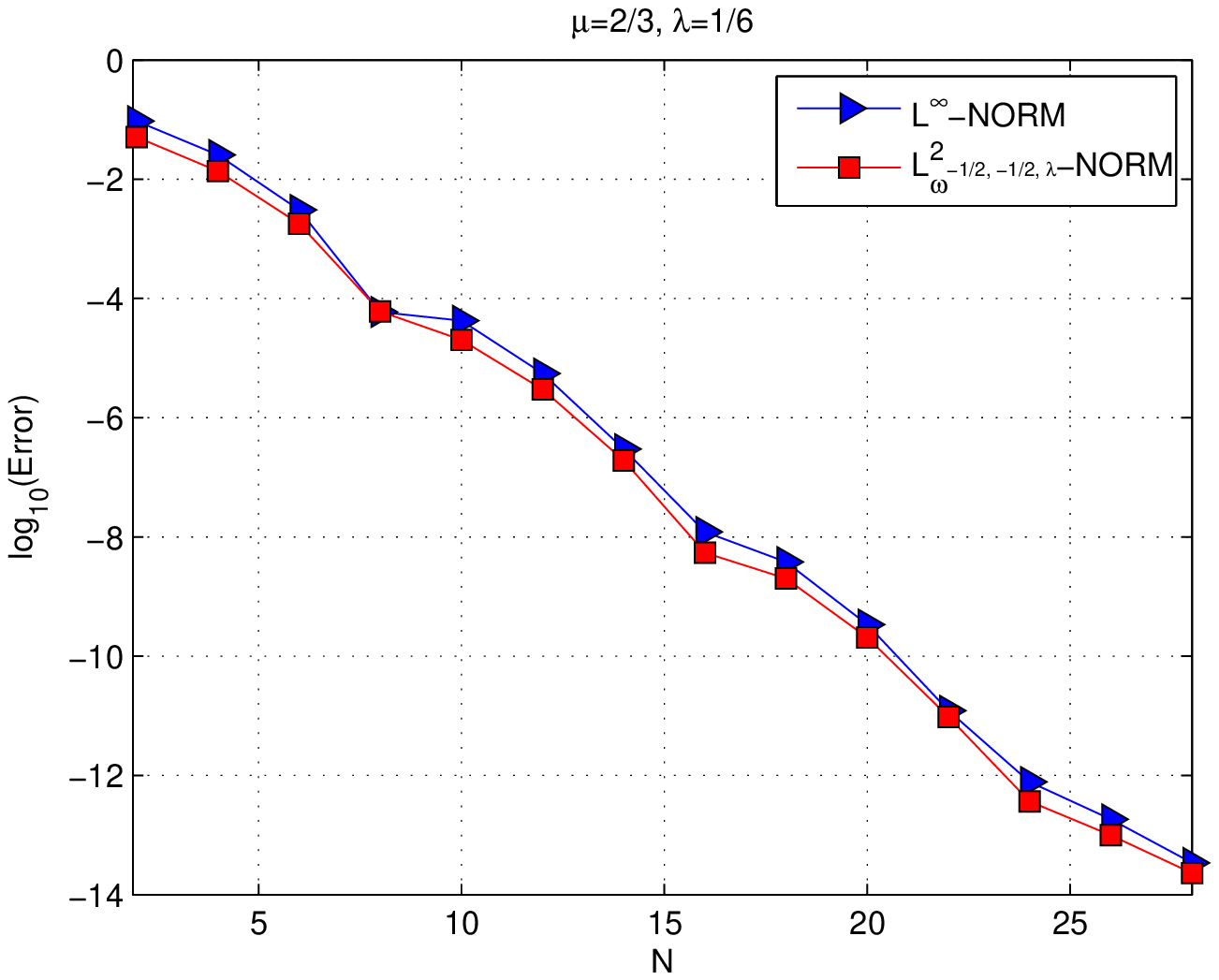}}
\centerline{(d)}
\end{minipage}
\caption{(Example \ref{pex3}) $L^{\infty}-$ and $L^{2}_{\omega^{-1/2,-1/2,\lambda}}-$norm errors versus $\lambda-$ polynomial degree $N$ with
(a) $\mu=0.5,\lambda=\frac{1}{2}$;
(b) $\mu=0.5,\lambda=\frac{1}{4}$;
(c) $\mu=\frac{2}{3},\lambda=\frac{1}{3}$;
(d) $\mu=\frac{2}{3},\lambda=\frac{1}{6}$;
}\label{fig3}
\end{figure*}

\begin{example}\label{emp1}
 Still consider a fabricated solution of \eqref{eq1} with $K(x,s)=1$ and
{\color{black}the exact solution as follows: $(i)$ $u(x)=x^{\gamma_{1}}+x^{\gamma_{2}}$; $(ii)$ $u(x)=\sin(x^{\gamma_{1}}+x^{\gamma_{2}}).$}
\end{example}

This solution, which looks quite unrealistic, is served as a good example to test the capability of
the proposed method in difficult situation. It is clear that for two general parameters $\gamma_1$ and
$\gamma_2$, it is not always possible to make $u(x^{1/\lambda})$ smooth. However, a careful examination
shows that $u(x^{1/\lambda})\in B^{2\gamma/\lambda+\beta+1-\epsilon}_{\omega^{\alpha,\beta,1}}(I)$ for any $\epsilon>0$, where
\begin{equation*}
\gamma=\begin{cases}
\begin{array}{r@{}l}
&\gamma_{1},\ \ \ \ \ \ \gamma_{1}~ is~ not~ an~ integer~ and~\gamma_{2}~ is~ an~ integer,\\
&\gamma_{2},\ \ \ \ \ \ \gamma_{1}~ is~ an~ integer~ and~\gamma_{2}~ is~not~ an~ integer,\\
&\infty,    \ \ \ \ \ \ \gamma_{1}~and~\gamma_{2}~are~both~integer,\\
&\min\{\gamma_{1},\gamma_{2}\},\ \ \ \ \ \ others.
\end{array}
\end{cases}
\end{equation*}
The numerical errors in log-log scale for a number of different $\gamma_1, \gamma_2, \mu, \lambda$, and
$\beta$ are presented in Figure \ref{fig4}.
Since $K(x,s)$ is smooth, the second terms in the error estimates \eqref{ex4} and \eqref{equati4}
are expected to be negligible as compared with the first terms in these estimates for large $N$.
As a consequence, the result shown in the figure should reflect the impact of the regularity of
$u(x^{1/\lambda})$ on the error behavior.
To closely observe the error decay rates, the $N^{-(2\gamma/\lambda+\beta+1)}$ and
$N^{-(2\gamma/\lambda+\beta+1)+1/2}$ decay rates are also plotted in the figure.
We have the following observations from this test:
1) all the error curves are straight lines in this log-log representation,
which indicates the algebraic convergence. This is consistent with the limited regularity
of $u(x^{1/\lambda})$;
2) the convergence rate is in a good agreement with the theoretical prediction
given in \eqref{ex4} and \eqref{equati4},
i.e., order $N^{-(2\gamma/\lambda+\beta+1)+1/2}$ and $N^{-(2\gamma/\lambda+\beta+1)}$  respectively;
3) decreasing $\lambda$ increases the regularity of $u(x^{1/\lambda})$, thus results in improvement of
the convergence.

\begin{figure*}[htbp]
\begin{minipage}[t]{0.48\linewidth}
\centerline{\includegraphics[scale=0.56]{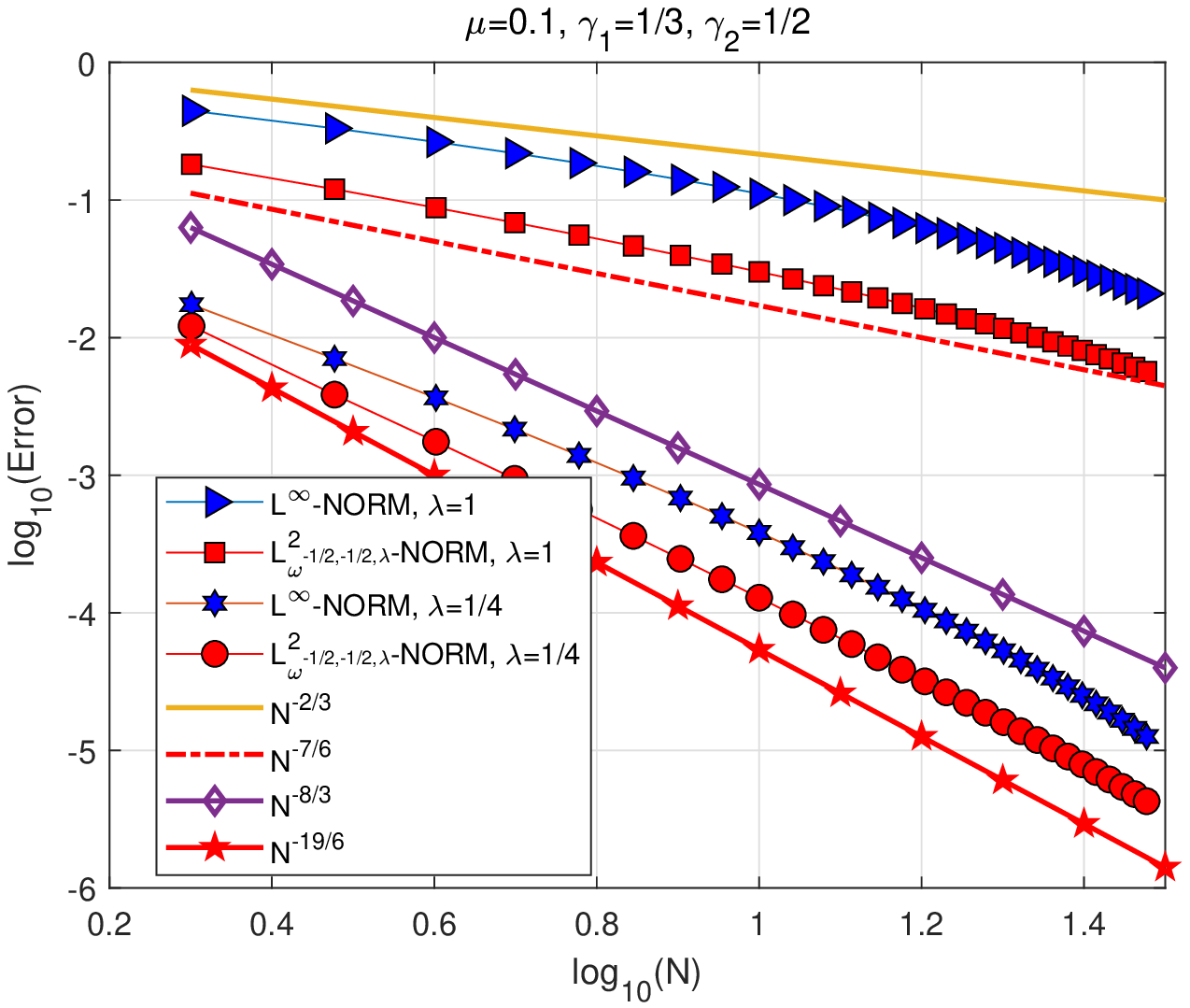}}
\centerline{(a)}
\end{minipage}
\begin{minipage}[t]{0.48\linewidth}
\centerline{\includegraphics[scale=0.56]{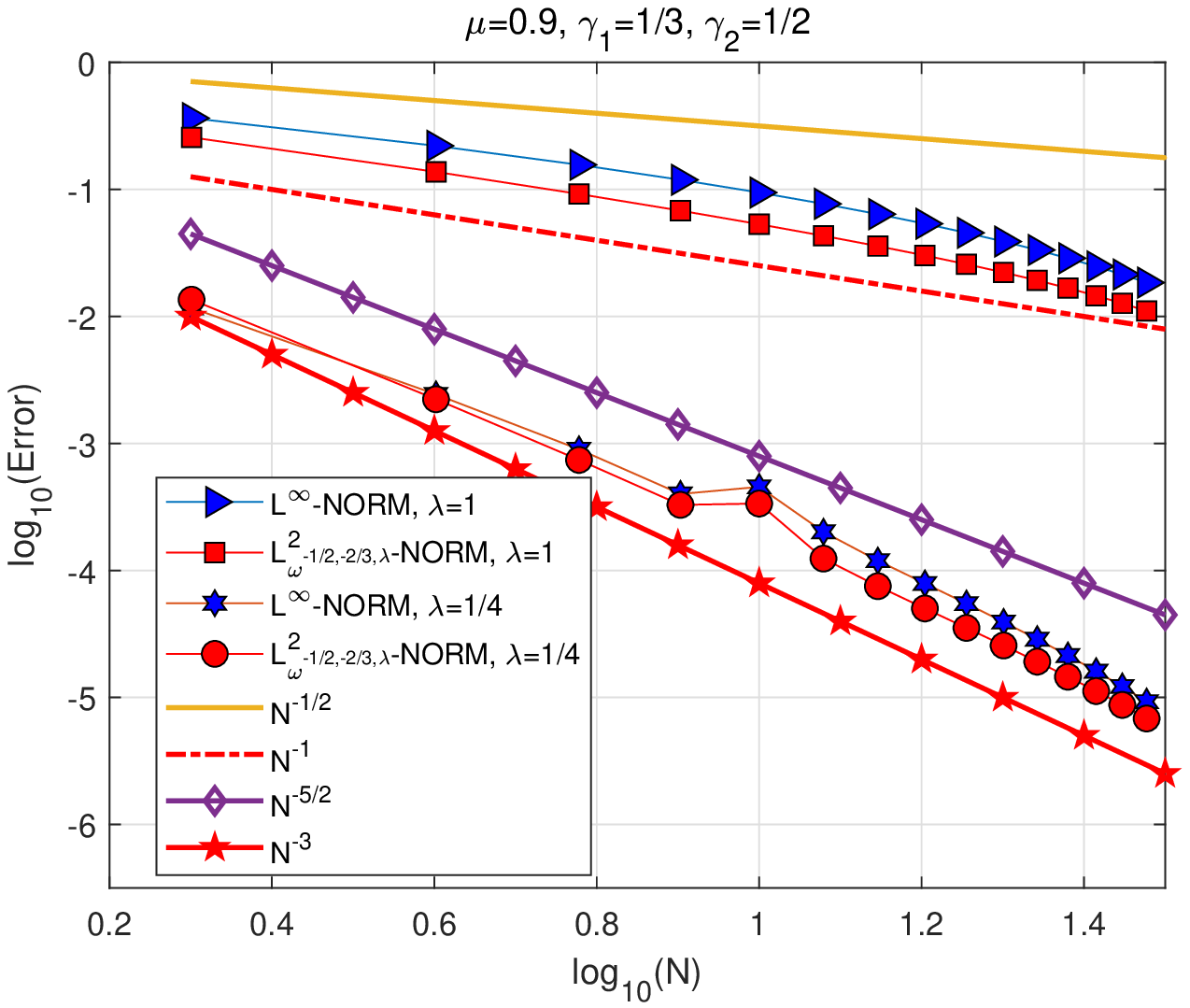}}
\centerline{(b)}
\end{minipage}
\vskip 3mm
\begin{minipage}[t]{0.48\linewidth}
\centerline{\includegraphics[scale=0.56]{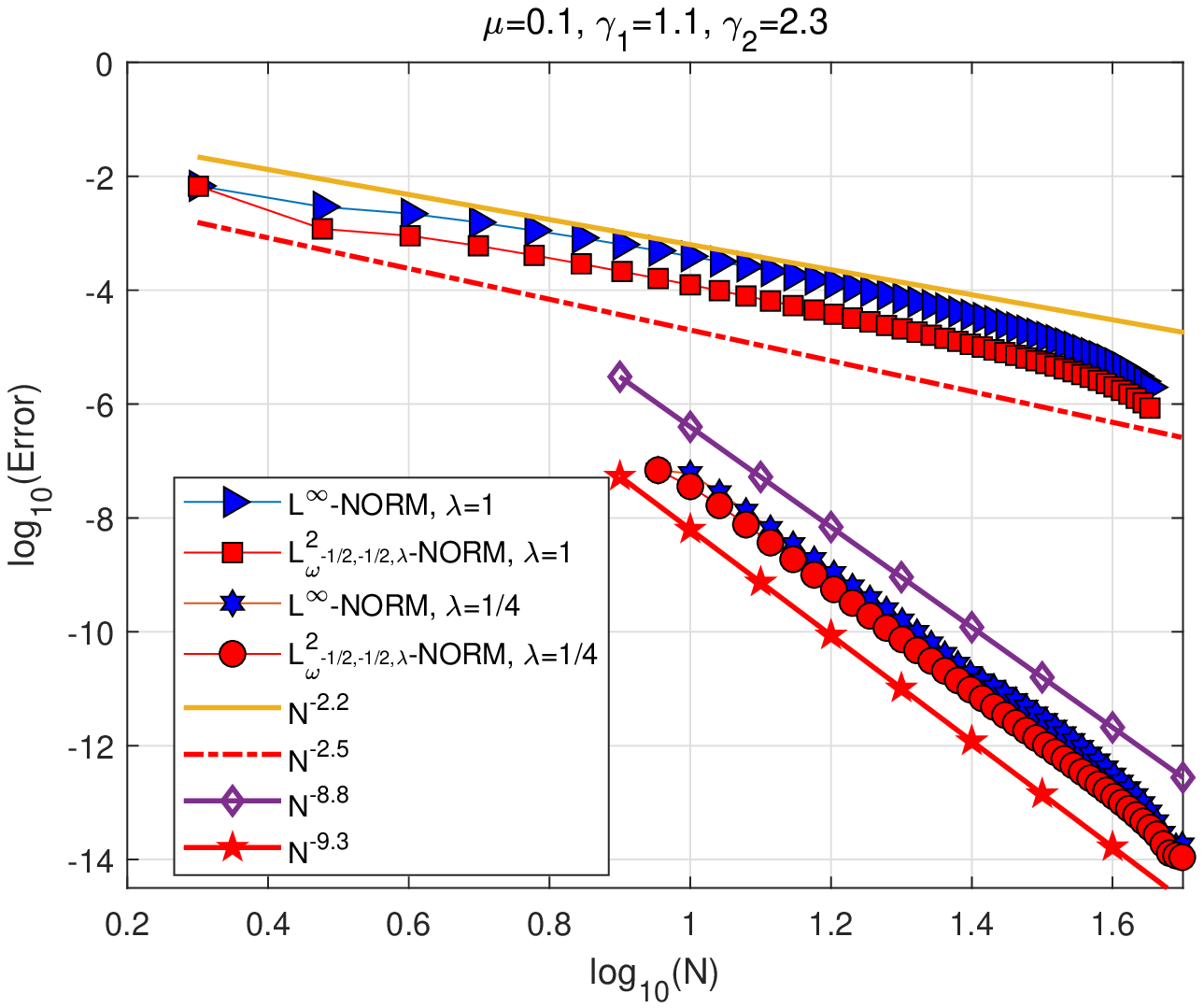}}
\centerline{(c)}
\end{minipage}
\begin{minipage}[t]{0.48\linewidth}
\centerline{\includegraphics[scale=0.56]{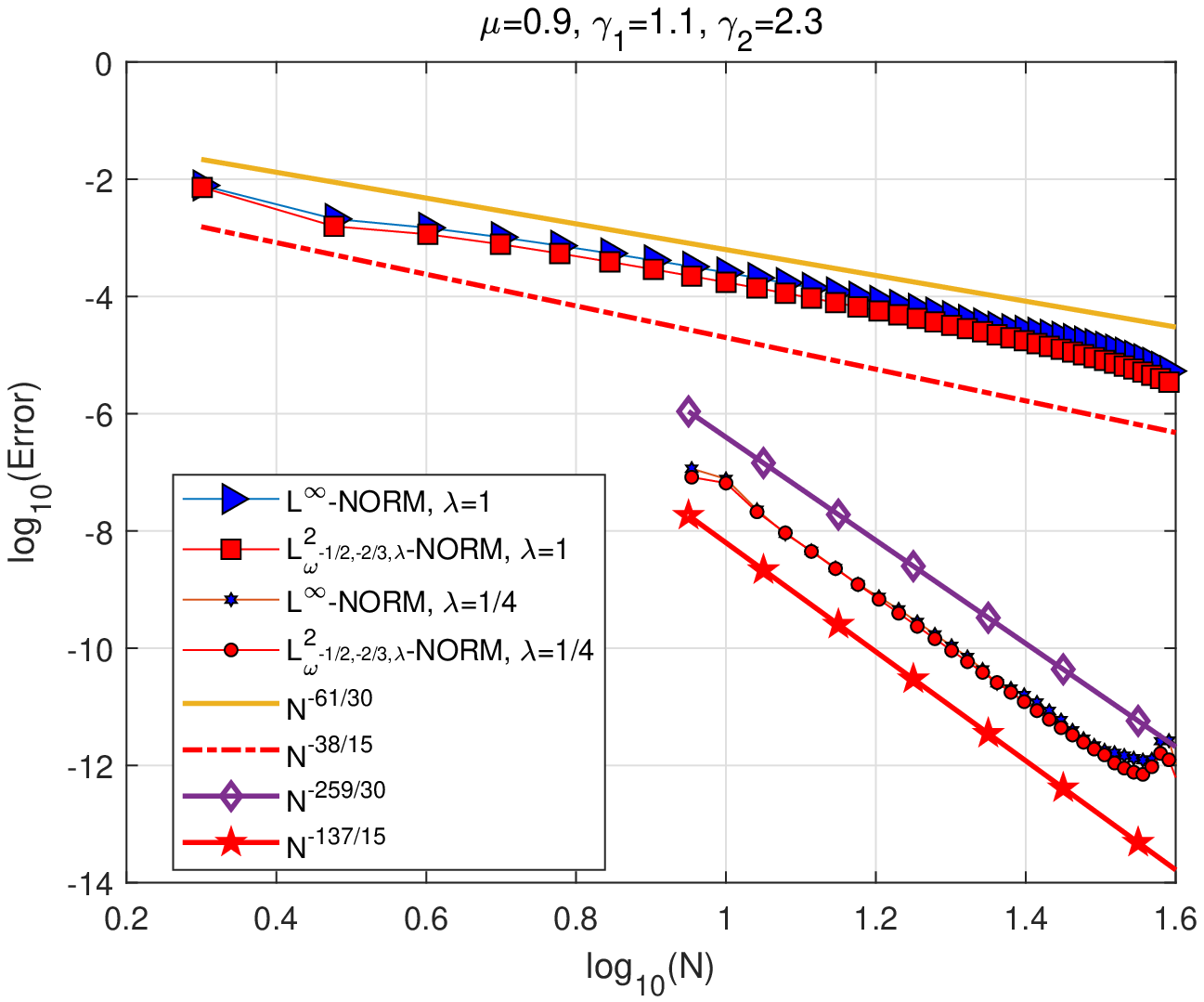}}
\centerline{(d)}
\end{minipage}
\vskip 3mm
\begin{minipage}[t]{0.48\linewidth}
\centerline{\includegraphics[scale=0.56]{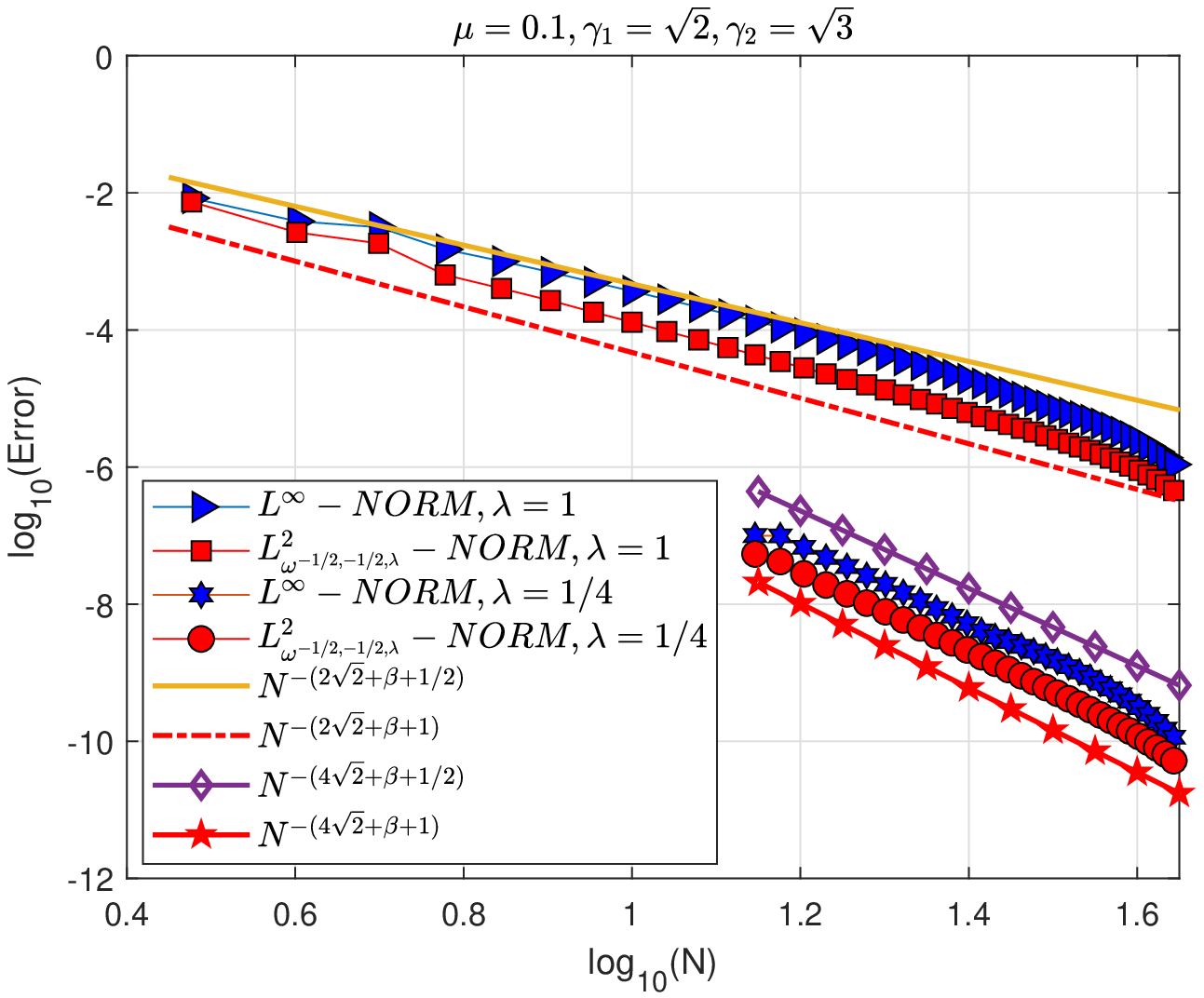}}
\centerline{(e)}
\end{minipage}
\begin{minipage}[t]{0.48\linewidth}
\centerline{\includegraphics[scale=0.56]{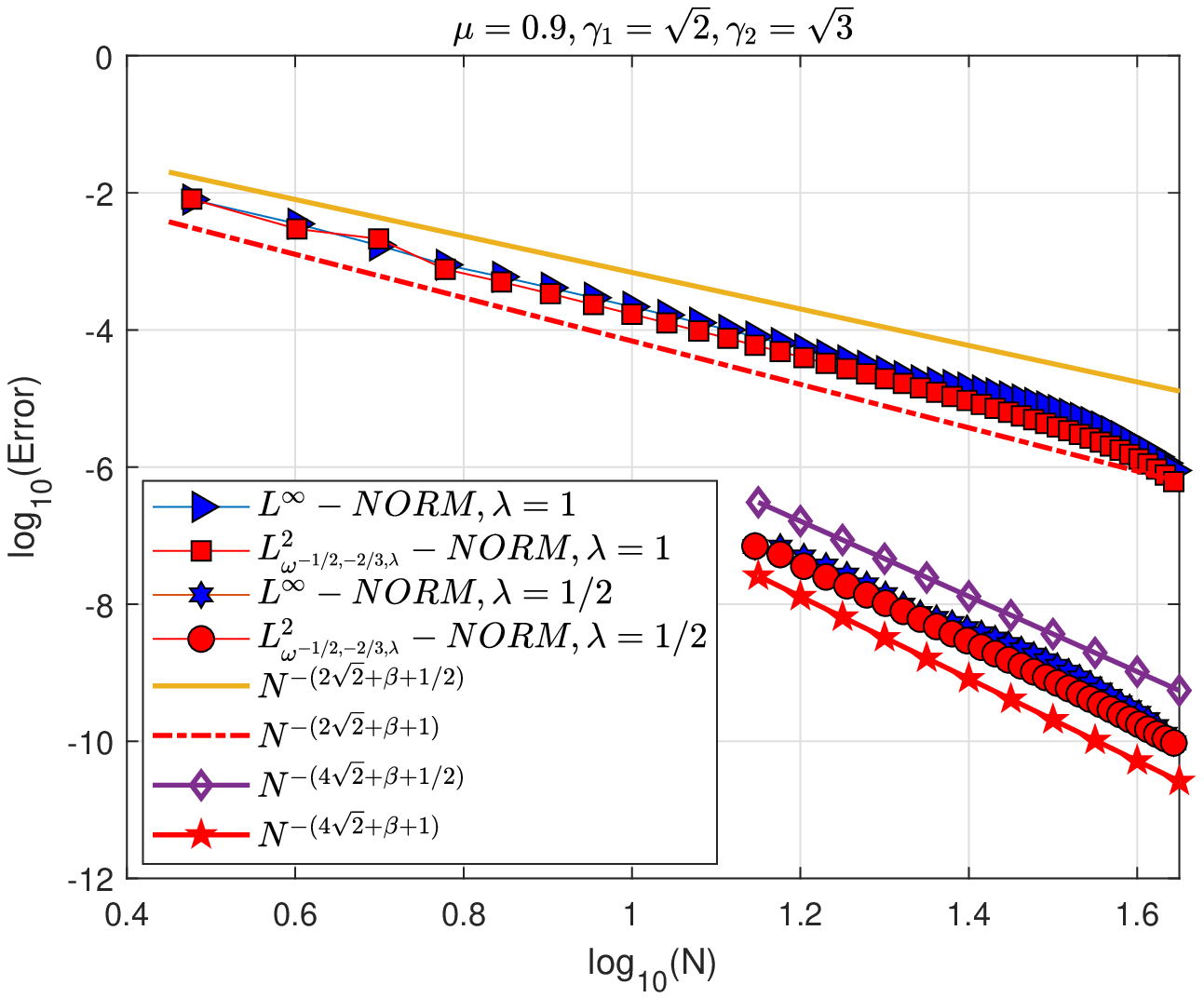}}
\centerline{(f)}
\end{minipage}

\caption{(Example \ref{emp1}) Errors versus $\lambda-$polynomial degree $N$ for
(i) $u(x)=x^{1/3}+x^{1/2}$:
(a) $\mu=0.1,\lambda=1$ or $1/4, \alpha=\beta=-1/2$;
(b) $\mu=0.9,\lambda=1$ or $1/4, \alpha=-1/2, \beta=-2/3$.
(ii) $u(x)=x^{1.1}+x^{2.3}$:
(c) $\mu=0.1,\lambda=1$ or $1/4, \alpha=\beta=-1/2$;
(d) $\mu=0.9,\lambda=1$ or $1/4, \alpha=-1/2, \beta=-2/3$.
(iii) $u(x)=\sin(x^{\sqrt{2}}+x^{\sqrt{3}})$:
(e) $\mu=0.1,\lambda=1$ or $1/2, \alpha=\beta=-1/2$;
(f) $\mu=0.9,\lambda=1$ or $1/2, \alpha=-1/2, \beta=-2/3$.
}\label{fig4}
\end{figure*}

\section{Concluding remarks}

In this work we investigated numerical solutions of the Volterra integral equations
with the weakly singular kernel $(x-s)^{-\mu}$, $0<\mu<1$.
The main difficulty in approximating this equation is that
the derivative of the solution is singular. This has resulted in low order convergence for any methods
using the traditional polynomials.
To overcome this difficulty, we established a framework of approach based on the fractional
Jacobi polynomials for the Volterra integral equations. Precisely,
a fractional Jacobi spectral-collocation method is constructed and analyzed for the underlying equation.
The significance of the approach is that it makes use of the fractional polynomials as the approximation
space, which can well capture typical solution structures of the singular integral equations.
It was known that a typical solution of the Volterra integral equations
with weakly singular kernel behaves like a series of power functions
$x^\gamma$. We have seen from the presented analysis that such functions can be much better approximated
by the space span$\{1, x^\lambda, x^{2\lambda}, \dots, x^{N\lambda}\}$ than the traditional polynomial space
span$\{1, x, x^2, \dots, x^N\}$ as long as $\lambda$ is suitably chosen.
In order to carry out a rigorous error
analysis, we first
established some approximation results for the weighted projection and interpolation operators. Then
we derived the error estimates in the $L^{\infty}-$ and weighted $L^{2}-$norms for
the proposed method.
A series of numerical examples were carried out to verify the theoretical claims.
The most remarkable property of the new method is its capability to
achieve spectral convergence for solutions of limited regularity.
{\color{black} It is worth mentioning that
the choice of $\lambda$ is also of importance for the efficiency of the new method.
Although there does not exist optimal choice for $\lambda$ for general problems,
it can be made according to the following strategy:
in case the regularity of the exact solution is unavailable,
the parameter $\lambda$ can be taken like $1/q$ with moderately large integer $q$ so that
$u(x^{q})$ is smooth enough. Our numerical experiments have shown that doing this
can increase the convergence rate about $q$ times.
}

\bibliographystyle{plain}
\bibliography{ref}
\end{document}